\documentclass[12pt]{article}
\usepackage[utf8]{inputenc}

\usepackage{preprint-layout} 

\usepackage{preprint-notation}

\usepackage{doi}

\DeclareMathOperator{\Trace}{tr}

\title{Cut Topology Optimization for Linear Elasticity with Coupling to Parametric Nondesign Domain Regions}
\author{Erik Burman, Daniel Elfverson, Peter Hansbo,\\
Mats G. Larson and Karl Larsson
}
\date{}

\begin{document}

\maketitle

\begin{abstract} We develop a density based topology optimization method for linear elasticity based on the cut finite element method. More precisely, the design domain is discretized using cut finite elements which allow complicated geometry to be represented on a structured fixed background mesh. The geometry of the design domain is allowed to cut through the background mesh in an arbitrary way and certain stabilization terms are added in the vicinity of the cut boundary, which guarantee stability of the method. Furthermore, in addition to standard Dirichlet and Neumann conditions we consider interface conditions enabling coupling of the design domain to parts of the structure for which the design is already given. These given parts of the structure, called the nondesign domain regions, typically represents parts of the geometry provided by the designer. The nondesign domain regions may be discretized independently from the design domains using for example parametric meshed finite elements or isogeometric analysis. The interface and Dirichlet conditions are based on Nitsche's method and are stable for the full range of density parameters. In particular we obtain a traction-free Neumann condition in the limit when the density tends to zero.
\end{abstract}


\section{Introduction}

Topology optimization can be a powerful tool for engineers in their quest for designing components that are light, strong and durable. Most topology optimization procedures are very general in nature and give few restrictions on the final design. One way for the designer to incorporate preferred geometric features in the final design is by specifying the geometry of the part of the component that is to be optimized, i.e., the so-called \emph{design domain}, and by specifying the geometry of the parts of the component that are already known, i.e., the so-called \emph{nondesign domain regions}. Many topology optimization procedures are based on the natural idea of seeking an optimal \emph{density} or \emph{material distribution} within the design domain. For efficiency, such procedures commonly utilize structured computational grids.

\paragraph{Our Contribution.}
In this paper we develop a flexible topology optimization method for linear elasticity, which supports design domain and nondesign domain regions with complex geometries. The method is based on the combination of a well established density based topology optimization approach on structured grids and the cut finite element method (CutFEM) for solving the linear elasticity problem. CutFEM is a fictitious domain method that allows us to apply the optimization algorithm without meshing the geometry of the design domain. In particular, it facilitates the use of a structured grid for computations, regardless of the geometry of the design domain.
Each subdomain, i.e., design domain or nondesign domain region, is equipped with its own mesh and finite element space, which can be independently chosen. These finite element spaces are weakly coupled using a weighted Nitsche's method, devised such that the interface condition transforms into a traction-free boundary condition wherever the material density on either side of the interface approaches zero.
The key features of this cut topology optimization method are:
\begin{itemize}
\item In the design domain the solution is represented using CutFEM, which facilitates the use of a fixed background mesh suitable for density based topology optimization in combination with a geometrically complex design domain boundary.
\item The design domain is coupled to given nondesign domain regions that may be discretized independently to the design domain using standard unstructured or parametric finite elements, isogeometric analysis or CutFEM.
\item The couplings between the various parts of the domain are weakly imposed using a weighted Nitsche's method where the weights depend on the material density such that Dirichlet or interface conditions locally turn into a traction-free boundary condition when the density approaches zero.
This coupling is stable, independent of the local density and provides a method of optimal order.
\end{itemize}

\paragraph{Previous Work.}
Ever since the seminal work of Bends{\o}e and Kikuchi \cite{Bendsoe1988}, density based topology optimization has been a rapidly developing field. Today, this fundamental technology is applied to a broad range of applications of industrial interest, including linear and nonlinear elasticity \cite{MR3038926,MR3340167,Bruns2001,Sigmund1997}, fluid-structure interaction \cite{MR3047899,doi:10.1002/nme.2777,MR3843862}, acoustics \cite{MR3406629,MR2947650,MR3240939} and electromagnetics \cite{MR3395903,5617251,6750741}. For more in-depth reviews of the field of topology optimization and its applications, see \cite{MR2008524,MR3182450,Sigmund2013}.

The development of the cut finite element method (CutFEM) stems from the classical work of Nitsche \cite{nitsche1971} for weak imposition of Dirichlet boundary conditions, which was later used by Hansbo and Hansbo \cite{HaHa2002} to formulate a fictitious domain method. Adding to this certain consistent stabilization terms, so-called \emph{ghost penalty} stabilization \cite{Burman2010}, yields CutFEM, a fictitious domain method based on a solid mathematical foundation. Regardless of how badly the boundary of the domain cuts the computational mesh, CutFEM is proven to be stable, to be of optimal order accuracy and to produce a well conditioned linear system of equations, see \cite{Burman2012,Burman2015,HaLaLa18}.

In shape and topology optimization CutFEM has previously been applied to problems where the geometry is represented using a level-set that is updated either using an optimization approach or an evolution equation, see \cite{Burman2017,Burman2018,MR3646361,BerWadBer2018}. The present work is however the first contribution where CutFEM is combined with a density based topology optimization approach.
To preserve the detail available in the higher-order cut finite element spaces when combined with a piecewise constant density approxmation we employ a multi-resolution strategy, such as previously explored in e.g. \cite{MR3670454,MR2878673}.

\paragraph{Outline.} In Section~2 we formulate the governing equations and the optimization model problem; in Section~3 we introduce the cut finite element method and prove stability results; 
in Section~4 we outline the topology optimization procedure; and in Section~5 we present some numerical examples.

\section{The Design Problem} \label{section:the-design-problem}
\begin{figure}\centering
\includegraphics[width=0.45\linewidth]{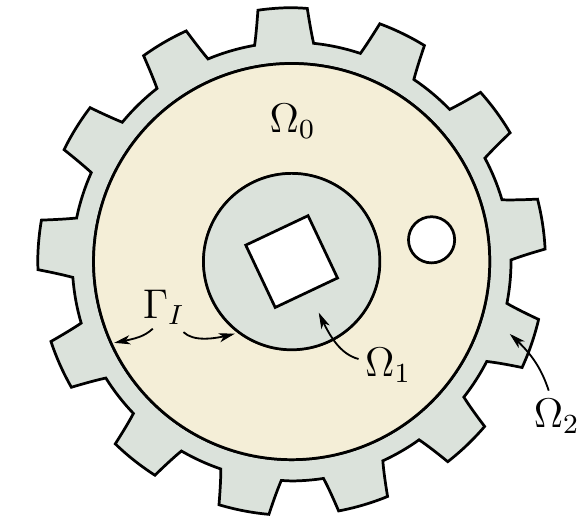}
\caption{Illustration of a cogwheel where the design domain $\Omega_0$ is to be optimized and the nondesign domain regions $\Omega_1$ and $\Omega_2$ are already supplied by the designer.
The domain of all interfaces between the subdomains $\{\Omega_i\}_{i=0}^2$ is denoted $\Gamma_I$.
}
\label{fig:domain}
\end{figure}

\paragraph{The Domain.}
We consider domains with the following structure:
\begin{itemize}
\item Let $\Omega$ be a domain in $\mathbb{R}^d$, $d=2$ or $3$, with a piecewise smooth boundary $\partial\Omega$ consisting of two disjoint parts
\begin{equation}
\partial\Omega = \Gamma_D \cup \Gamma_N
\end{equation}
where $\Gamma_D$ and $\Gamma_N$ are the Dirichlet and Neumann parts of the boundary, respectively.

\item Let $\Omega$ have the following non-overlapping decomposition into subdomains
\begin{equation} \label{eq:decomposition}
\Omega = \Omega_0 \cup \Omega_1 \cup \cdots \cup \Omega_n
\end{equation}
where $\Omega_0$ is the design domain and $\{\Omega_i\}_{i=1}^n$ are
the nondesign domain regions.

\item
The domain of all interfaces in the decomposition \eqref{eq:decomposition} is defined
\begin{equation}
\Gamma_I = \bigcup_{\substack{i,j=0\\i < j}}^n \partial\Omega_i \cap \partial\Omega_j
\end{equation}

\end{itemize}

\paragraph{Linear Elasticity.} We assume that the physics in our problem is governed by linear elasticity and we formulate this as the following interface problem on $\Omega$:
\begin{itemize}
\item Let the stress and strain tensors $\sigma$ and $\epsilon$ be defined by
\begin{align}
\sigma(v) &= 2 \mu \epsilon(v) 
+  \lambda \text{tr}(\epsilon(v))I
\\
\epsilon(v) &= \frac{1}{2}( v \otimes\nabla + \nabla \otimes v )
\end{align}
where $\mu > 0$ and $\lambda \geq 0$ are the Lam\'e parameters of the material and $a \otimes b$ denotes the tensor 
product between vectors $a\in \IR^d$ and $b\in \IR^d$ with elements $(a \otimes b)_{kl} = a_k b_l$.
We allow $\mu$ and $\lambda$ to vary over $\Omega$ and assume that $\lambda/\mu \lesssim 1$, which means that the material does not approach incompressibility.

\item The interface problem reads: Find the displacement field
$u:\Omega \rightarrow \IR^d$ such that
\begin{alignat}{2}\label{eq:standard-a}
-\sigma (u) \cdot \nabla &= f \qquad 
&& \text{in $\Omega\setminus \Gamma_I$}
\\ \label{eq:standard-c}
u &= 0 \qquad && \text{on $\Gamma_D$}
\\ \label{eq:standard-b}
\sigma(u) \cdot n &= g_N \qquad && 
\text{on $\Gamma_N$}
\\ \label{eq:standard-d}
[u] &= 0 &&  \text{on $\Gamma_I$}
\\ \label{eq:standard-e}
[\sigma(u)\cdot n] &= 0 &&  \text{on $\Gamma_I$}
\end{alignat}
with given data $\mu,\lambda : \Omega \rightarrow \IR$, $f : \Omega \rightarrow \IR^d$ and $g_N : \Gamma_N \rightarrow \IR^d$. Here $[\cdot]$ denotes the jump over the interface which we will define more explicitly in Section~\ref{section:method}.

\item In variational form this problem reads: Find $u \in V = \{ v \in H^1(\Omega) : v |_{\Gamma_D} = 0\}$ 
such that 
\begin{equation}\label{eq:weak}
a(u,v) = l(v)\qquad \forall v \in V
\end{equation}
where the forms are defined by
\begin{equation}\label{eq:forms}
a(v,w) = (\sigma(v),\epsilon(w))_\Omega \, , \qquad l(v) = (f,v)_\Omega + (g_N,v)_{\Gamma_N}
\end{equation}
with $(c,d)_\omega = \int_\omega c \cdot d\,$ denoting the standard $L^2(\omega)$ inner product, which induce the $L^2(\omega)$ norm $\| c \|_\omega^2 = (c,c)_\omega$.

Assuming there exists some lower bound $0 < \mu_\mathrm{min} \leq \mu$ on $\Omega$, and that $f,g_N$ are sufficiently regular it follows from the Lax--Milgram lemma that \eqref{eq:weak} has a unique solution 
$u \in V$.
\end{itemize}

\paragraph{The Optimization Problem.} We follow the standard approach of using a density function for representing where in the design domain $\Omega_0$ we have material. Our optimization problem is defined as follows:
\begin{itemize}
\item Let $\chi:\Omega\to[0,1]$ be density function specifying where in $\Omega$ we have material. In the nondesign domain regions $\Omega\setminus\Omega_0$ we have $\chi=1$ whereas in the design domain $\Omega_0$ the density function $\chi$ will be determined through the optimization procedure.
We define the density spaces
\begin{align}
W_0 &= \left\{ \chi : \overline{\Omega}_0 \to [\chi_\mathrm{min},1] \, : \, \chi \in L^\infty(\overline{\Omega}_0) \right\}
\\
W_i &= \left\{ \chi : \overline{\Omega}_i \to 1 \right\} \quad \text{for $i \geq 1$}
\\
W &=  W_0 \oplus W_1 \oplus \dots \oplus W_n
\end{align}
where $\chi_\mathrm{min}$ is some very small constant.

\item We define the data in the interface problem \eqref{eq:weak} such that it scales with $\chi$, i.e.,
\begin{equation} \label{eq:data-scaling}
\mu = \chi\widehat{\mu} \,, \qquad
\lambda = \chi\widehat{\lambda} \,, \qquad
f = \chi\widehat{f} \,, \qquad
g_N = \chi\widehat{g}_N
\end{equation}
where the hat versions of these functions are the actual data supplied in the set-up of the optimization problem. As both Lamé-parameters scales with $\chi$ it clearly holds that $\widehat{\mu} > 0$, $\widehat{\lambda} \geq 0$ and $\widehat{\lambda}/\widehat{\mu}\lesssim 1$.

\item Let $\Delta_\mathrm{vol}$ be the target proportion of material volume in the final design within $\Omega_0$ corresponding to a target final design volume of $\Delta_\mathrm{vol} | \Omega_0 |$.
This imposes the following volume constraint on $\chi$
\begin{equation}\label{eq:vol-constraint}
\int_{\Omega_0} \chi = \Delta_\mathrm{vol} |\Omega_0|
\end{equation}

\item Provided with some objective functional $J(\chi)=J_\chi(u_\chi)$, where $u_\chi$ is the solution to \eqref{eq:weak} given $\chi$,
we pose our optimization problem as seeking the minimizing density
\begin{equation} \label{eq:min}
  \chi = \arg \left(\min_{\chi\in W} J_\chi(u) \right)
\end{equation}
under the constraints that $u$ solves \eqref{eq:weak} and that $\chi$ satisfies \eqref{eq:vol-constraint}.

\item
The optimization problem \eqref{eq:min} can also be expressed using a Lagrangian formulation
where we seek a critical point $(\chi,u,p,\eta)$ to
\begin{equation} \label{eq:lagrangian}
  \mcL(\chi,u,p,\eta) = J_\chi(u) - a_\chi(u,p) + l_\chi(p) - \eta\left(\int_{\Omega_0} \chi -\Delta_\mathrm{vol}|\Omega_0|\right) 
\end{equation}
where $p\in V$ and $\eta \in \IR$ are Lagrange penalty parameters used to enforce the constraints on $u$ and $\chi$.

Here we use the subscript $\chi$ notation to emphasize that a form depend on $\chi$, which in the case of $a$ and $l$ is due to the given data being scaled by $\chi$, c.f. \eqref{eq:data-scaling}.
The critical point can be determined by solving the problem: Find $\chi\in W$, $u\in V$, $p\in V$ and $\eta\in\mathbb{R}$ such
that
\begin{align}
  \partial_\eta\mcL &= 0
  \label{eq:dLdlambda}
  \\
  \langle\partial_p\mcL,\delta p\rangle &= 0\qquad\forall \delta p \in V
  \label{eq:dLdq}
  \\
  \langle\partial_u\mcL,\delta u\rangle &= 0\qquad\forall \delta u \in V 
  \label{eq:dLdv}
  \\
  \langle\partial_\chi\mcL,\delta \chi\rangle &= 0\qquad\forall \delta \chi \in W
  \label{eq:dLdchi}
\end{align}
where $\langle\partial_p\mcL(\chi,u,p,\eta),\delta p\rangle$ denotes the partial derivative of $\mcL$ with respect to $p$ in the direction of $\delta p$, cf. \cite{AJT04}.
The first equation \eqref{eq:dLdlambda} gives the volume constraint \eqref{eq:vol-constraint}.
From \eqref{eq:dLdq} we get the primal problem: Find $u\in V$ such that
\begin{equation} \label{eq:primal-prob}
a_\chi(u,\delta p) = l_\chi(\delta p) \qquad\forall \delta p \in V
\end{equation}
so at the critical point $u=u_\chi$, i.e., with $\chi$ given $u$ is the solution to \eqref{eq:weak}.
Analogously, from \eqref{eq:dLdv} we get the dual problem: Find $p\in V$ such that
\begin{equation} \label{eq:dual-prob}
  a_\chi(\delta u,p) = \langle\partial_u J_\chi , \delta u \rangle \qquad \forall \delta u\in V
\end{equation}

\item To solve \eqref{eq:min} we will use a steepest descent type algorithm and thus, we must be able to compute the derivative of $J(\chi) = J_\chi(u_\chi)$ with respect to $\chi$, i.e.,
$\langle d_\chi J,\delta\chi \rangle$,
where we use the notation $d_\chi$ to emphasize that this is the total derivative.
For $\chi$ satisfying the volume constraint \eqref{eq:vol-constraint} the total derivative can be expressed as the following, more easily evaluated, partial derivative of the Lagrangian
\begin{equation} \label{eq:total-to-partial}
\langle d_\chi J ,\delta\chi \rangle \bigr|_\chi
=
\langle \partial_\chi \mcL,\delta\chi \rangle \bigr|_{(\chi,u_\chi,p_\chi,\eta_\chi)}
\end{equation}
where $u_\chi$ and $p_\chi$ are the solutions to the primal and dual problems and the value of $\eta_\chi$ is
chosen in such a way that the partial derivative on the right hand side is tangent to the space of density functions satisfying the volume constraint \eqref{eq:vol-constraint}.
In Section~\ref{section:top-opt} below we outline such a steepest descent type algorithm.

\item 
In the present work we as our objective functional choose compliance, i.e.,
$J_\chi(u)=l_\chi(u)$,
whereby the dual problem \eqref{eq:dual-prob} becomes
\begin{equation}
  a_\chi(\delta u,p) = l_\chi( \delta u ) \quad \forall \delta u\in V
\end{equation}
Due to the symmetry of $a_\chi$ the primal and dual problems in this case are the same and thus the critical point will satisfy $u=p=u_\chi$.
The optimization problem \eqref{eq:min} can thereby be expressed as seeking the critical point $(\chi,u,\eta)$ characterized by
\begin{equation} \label{eq:compliance-lagrangian}
  \mcL(\chi,u,\eta) = \underbrace{-2\left( \frac{1}{2}a_\chi(u,u) - l_\chi(u)\right)}_{\Pi_\chi(u)} - \eta
\biggl(
\underbrace{
\int_{\Omega_0} \chi - \Delta_\mathrm{vol}|\Omega_0|
}_{\Lambda_\chi}
\biggr)
\end{equation}
where we recognize $\Pi_\chi(u)$ as a Ritz functional
and we denote the volume constraint functional by $\Lambda_\chi$.

\end{itemize}

\section{The Cut Finite Element Method}

\subsection{The Mesh and Finite Element Spaces}

\paragraph{Cut Mesh on Design Domain.} On the design domain $\Omega_0$ we will employ cut finite elements that allow the design domain to arbitrarily intersect the mesh. We here define our cut mesh and cut finite element space:

\begin{itemize}
\item Let $\widetilde{\Omega}_0 \subset \mathbb{R}^d$ be a polygonal domain such that 
$\Omega_0 \subset \widetilde{\Omega}_0$ and let 
$\{\widetilde{\mcT}_{0,h} \,, \ h \in (0,h_{0}]\}$ be a family of quasiuniform 
partitions with mesh parameter $h$ of $\widetilde{\Omega}_0$ into 
shape regular elements $T$. We denote $\widetilde{\mcT}_{0,h}$ as the background mesh of $\Omega_0$.

\item We define the active mesh
\begin{equation} \label{eq:active-mesh}
\mcT_{0,h} = \{T \in \widetilde{\mcT}_{0,h} \, : \, T \cap 
\Omega_0 \neq \emptyset \} 
\end{equation}
consisting of all elements in $\widetilde{\mcT}_{0,h}$ with a non-zero intersection with $\Omega_0$.

\item Let $V_{0,h}$ be a space of $\IR^d$ 
valued continuous piecewise polynomials or tensor product polynomials of order $p$ defined on the active mesh $\mcT_{0,h}$. In particular, for our numerical examples, we use tensor product B-splines of maximum regularity with polynomial order $p=2$, which yields a finite element space $V_{0,h}|_{\Omega_0} \subset C^1(\Omega_0)$.
\end{itemize}

\paragraph{Discrete Density on Design Domain.}
The discretization of the density on the design domain $\Omega_0$ will be piecewise constant albeit on a refined mesh compared to \eqref{eq:active-mesh}. We define our discrete density space as follows:
\begin{itemize}
\item Let $\widetilde{\mcT}_{0,h/2^k}$ be the refinement of the background mesh $\widetilde{\mcT}_{0,h}$ constructed by uniformly splitting each element into $2^d$ elements $k$ times. We define the active $k$-refined mesh
\begin{equation} \label{eq:k-ref-mesh}
\mcT_{0,h/2^k} = \{T \in \widetilde{\mcT}_{i,h/2^k} \, : \, T \cap 
\Omega_0 \neq \emptyset \} 
\end{equation}

\item Let $W_{0,h,k}$ be the space of piecewise constant scalar functions on the active $k$-refined mesh $\mcT_{0,h/2^k}$. The discrete density space on $\Omega_0$ is given by the restriction $W_{0,h,k}|_{\Omega_0}$ and clearly $W_{0,h,k}|_{\Omega_0} \subset W_0$.
\end{itemize}

\paragraph{Parametric Meshes in Nondesign Domain Regions.}
In each nondesign domain region $\Omega_i$, $i=1,\dots,n$, we construct a mesh fitted to $\Omega_i$ via a parametric mapping $F_i$ from a reference domain $\widehat{\Omega}_i$ as follows:

\begin{itemize}
\item
Let $\widehat{\Omega}_i \subset \IR^d$ be a polygonal domain 
associated with a diffeomorphism $F_i^{-1} : \Omega_i \to \widehat{\Omega}_i$, i.e.,
the bijective mapping $F_i : \widehat{\Omega}_i \to \Omega_i$ is a differentiable function.

\item
Let $\{\widehat{\mcT}_{i,h} \,, \ h \in (0,h_{0}]\}$ be a family of quasiuniform 
partitions with mesh parameter $h$ of $\widehat{\Omega}_i$ into 
shape regular elements $T$.
%
On $\widehat{\mcT}_{i,h}$ we define $\widehat{V}_{i,h}$ to be a space of $\IR^d$ 
valued continuous piecewise polynomials or tensor product polynomials of order $p$.

\item
In the physical domain $\Omega_i$ we now define our finite element space as
\begin{equation}
V_{i,h} = \{ v \in H^1(\Omega_i) \, : \, v \circ F_i^{-1} \in \widehat{V}_{i,h}  \}
\end{equation}

\end{itemize}

\paragraph{The Complete Finite Element Space.}

The finite element space on the full domain $\Omega$ is defined as
  \begin{equation}
  V_h = \bigoplus_{i=0}^n V_{i,h}
  \end{equation}

\begin{rem} The use of a cut mesh on the design domain $\Omega_0$ and
(body-fitted) parametric meshes for the nondesign domain regions $\Omega\setminus\Omega_0$ is only for pedagogical reasons and not due to any limitation in the method. Any subdomain can be equipped with either a cut mesh or a parametric mesh, and actually the meshes can be concurrently cut and parametric, see \cite{JonLarLar17}.
\end{rem}

\subsection{The Method} \label{section:method}

\paragraph{Jump and Average Operators.}
On $\partial\Omega_i \cap \Gamma_D$ and on $\partial\Omega_i \cap \partial\Omega_j \in \Gamma_I$, $i < j$, we define the following jump, average and weighted average operators
\begin{alignat}{4}
[w] &= w_i \quad &\text{on $\Gamma_D$}&
\,,\qquad
& [w] &= w_i - w_j \quad&\text{on $\Gamma_I$}&
\\
\langle w \rangle &= w_i \quad&\text{on $\Gamma_D$}&
\,,\qquad
& \langle w \rangle &= \frac{w_i+w_j}{2} \quad&\text{on $\Gamma_I$}&
\\ \label{eq:weighted1}
\{ w \} &= w_i \quad&\text{on $\Gamma_D$}&
\,,\qquad\quad
& \{w\} &= \frac{\mu_j}{\mu_i+\mu_j} w_i + \frac{\mu_i}{\mu_i+\mu_j} w_j \quad&\text{on $\Gamma_I$}&
\end{alignat}
We define the normal on $\Gamma_I$ as $n=n_i=-n_j$ where $n_i$ and $n_j$ are the outward pointing boundary normals to $\Omega_i$ and $\Omega_j$, respectively.
For consistency the second terms in above operators then have the opposite sign on fluxes, for example
\begin{align}
\{\sigma\}\cdot n &= \frac{\mu_j}{\mu_i+\mu_j} \sigma_i\cdot n_i - \frac{\mu_i}{\mu_i+\mu_j} \sigma_j \cdot n_j = \{\sigma\cdot n\}
\end{align}

\paragraph{The Finite Element Method.}

Find 
$u_h \in V_h$ such that
\begin{equation}\label{eq:fem}
A_h(u_h ,v) = l_h(v) 
\qquad 
\forall v \in V_h
\end{equation}
where the forms are given by
\begin{align} \label{eq:Ah}
A_h(v,w) &= a_h(v,w) + s_h(v,w) + \beta b_h(v,w) + c_h(v,w)
\\ \label{eq:ah}
a_h(v,w) &= (\sigma(v),\epsilon(w))_{\Omega\setminus\Gamma_I}
\\ \label{eq:bh}
b_h(v,w) &= ( \{h^{-1} 2\mu\}[v],[w])_{\Gamma_I \cup \Gamma_D} + (\{h^{-1}\lambda\} [v\cdot n],[w \cdot n])_{\Gamma_I \cup \Gamma_D}
\\ \label{eq:consistency-term}
c_h(v,w) &= -(\{\sigma(v)\cdot n\}, [w])_{\Gamma_I \cup \Gamma_D} -([v],\{\sigma(w)\cdot n\})_{\Gamma_I \cup \Gamma_D}
\\ \label{eq:lh}
l_h(v) &= (f,v)_{\Omega} + (g_N,v)_{\Gamma_N}
\end{align}
where $\beta$ is a positive parameter and $s_h$ is a stabilization form which we outline below.

\begin{rem}
By keeping the mesh size and material parameters inside the averages we conveniently allow for subdomain wise choices of these parameters without cluttering the presentation.
\end{rem}

\begin{rem}
While the notation for the integrals used above is brief, it is convenient to pose them more explicitly when implementing the method. For example, the bulk integral can be stated
\begin{equation}
(\sigma(v),\epsilon(w))_{\Omega\setminus\Gamma_I}
= \sum_{i=0}^n (\sigma_i(v_i),\epsilon(w_i))_{\Omega_i}
\end{equation}
and as an example interface/boundary term we take
\begin{equation}
([v],[w])_{\Gamma_I \cup \Gamma_D} = ([v],[w])_{\Gamma_I} + ([v],[w])_{\Gamma_D}
\end{equation}
where the two integrals on the right can be written
\begin{align}
([v],[w])_{\Gamma_I} &= \sum_{i=0}^n \sum_{j=i+1}^n ([v],[w])_{\partial\Omega_i\cap\partial\Omega_j}
\\
([v],[w])_{\Gamma_D} &= \sum_{i=0}^n ([v],[w])_{\partial\Omega_i\cap\Gamma_D}
= \sum_{i=0}^n (v_i,w_i)_{\partial\Omega_i\cap\Gamma_D}
\end{align}
\end{rem}

\paragraph{The Stabilization Form.} The stabilization form
$s_h(v,w) = \sum_{i=0}^n s_{i,h}(v_i,w_i)$
must satisfy the following abstract properties on each subdomain $\Omega_i$, $i=0,\dots,n$:
\begin{itemize}
\item The form is consistent, i.e., it holds
\begin{equation}
s_{i,h}(v_i,v_i) =0 \,, \quad \forall v_i\in H^{p+1}(\Omega_i)
\end{equation}

\item The form satisfies the estimate
\begin{equation}
s_{i,h}(\pi_i v_i,\pi_i v_i) \lesssim h_i^{2p}\| v_i \|_{H^{p+1}(\Omega_{\mcT_i})}^2 \,, \quad \forall v_i\in H^{p+1}(\Omega_i)
\end{equation}
where $\pi_i:H^{p+1}(\Omega_{\mcT_i}) \to V_{i,h}$ is a suitably defined interpolant.

\item The following inverse inequality holds
\begin{equation} \label{eq:inverse-subdomain}
h_i (\sigma_i(v_i),\epsilon(v_i))_{\partial\Omega_i}
\lesssim
(\sigma_i(v_i),\epsilon(v_i))_{\Omega_i} + s_{i,h}(v_i,v_i)
 \,, \quad \forall v_i\in V_{i,h}
 \end{equation}
\end{itemize}
In the first two properties we assume an extension of $v_i \in H^{p+1}(\Omega_i) \to H^{p+1}(\Omega_{\mcT_i})$.

\begin{rem}[Choice of Stabilization Form] When the finite element space $V_{i,h}$ is fitted to $\Omega_i$ the abstract properties above are satisfied by the trivial choice
\begin{equation}
s_{i,h}(v,w)=0
\end{equation}
On the other hand, when $\Omega_i$ is allowed to cut the domain of the finite element space $V_{i,h}$ we instead choose the so called ghost penalty term
\begin{equation} \label{eq:stab-form}
s_{i,h}(v,w)= \sum_{k=1}^p \gamma_{i,k} h_i^{2 k - 1} ([\partial_n^k v],[\partial_n^k w])_{\mcF_{i,h}(\partial\Omega_i)}
\end{equation}
where $\mcF_{i,h}(\partial\Omega_i)$ is the domain of all interior faces in $\mcT_{i,h}$ belonging to elements cut by $\partial\Omega_i$, $[\partial_n^k v]$ is the jump in the $k$:th derivative in the face normal direction, and $\gamma_{i,k}$ are positive parameters which scale with $\widehat{\mu}$.
This form satisfies the abstract properties, see \cite{HaLaLa18}.

If a $H^2(\Omega_i)$ finite element space, for example quadratic B-splines, is used in the cut situation, an alternate approach to adding a stabilization form $s_{i,h}$ is to simply remove basis functions with small support inside the $\Omega_i$. This approach however requires some additional consistent least-squares terms to prove coercivity, see \cite{ElfLarLar18,2018arXiv180405654E}.
\end{rem}

\paragraph{Interface and Boundary Conditions when $\boldsymbol\chi \boldsymbol\to \boldsymbol 0$.}
For the weighted average \eqref{eq:weighted1} we have the properties
\begin{equation}
\{\mu \} = \frac{2\mu_i\mu_j}{\mu_i+\mu_j} \qquad\text{and}\qquad
\{\mu w\} = \frac{\mu_i\mu_j}{\mu_i+\mu_j}(w_i + w_j) = \{\mu\} \langle w \rangle
\end{equation}
and as $\mu=\chi\hat{\mu}$ we in the limit $\chi\to 0$ have
\begin{equation}
\lim_{\chi_i\to 0} \{\mu \} =
\lim_{\chi_j\to 0} \{\mu \} =
\lim_{(\chi_i,\chi_j)\to (0,0)} \{\mu \} = 0
\end{equation}
Also $\lambda$ scales with $\chi$ and we can write $\lambda = \mu \frac{\lambda}{\mu} = \mu \frac{\widehat{\lambda}}{\widehat{\mu}}$ where $\frac{\widehat{\lambda}}{\widehat{\mu}}$ is assumed bounded.
Thus, all integrands in the penalty \eqref{eq:bh} or in the consistency \eqref{eq:consistency-term} terms can be stated on the form $\{\mu w\}=\{\mu\}\langle w \rangle$ so the integrands will locally give zero contribution to the forms $b_h$ and $c_h$ wherever $\chi\to 0$.
Hence, the interface or Dirichlet conditions in parts of $\Gamma_I \cap \Gamma_D$ where $\chi\to 0$ on either side turns into a homogeneous Neumann condition (on both sides in the case of interfaces).

\subsection{Properties of $\boldsymbol{A}_{\boldsymbol{h}}$}

\paragraph{Norms.}
We define the energy norm
\begin{align}
\tn v \tn_h^2 &= \| v \|_{a_h}^2 + \| v \|_{b_h}^2 + \| v \|_{s_h}^2
\\&\quad \nonumber
+ \| \{ h^{-1} 2\mu \}^{-1/2} \{ 2\mu\epsilon(v) \cdot n \} \|_{\Gamma_I\cup\Gamma_D}^2
\\&\quad \nonumber
+ \| \{ h^{-1} \lambda \}^{-1/2} \{ \lambda\Trace(\epsilon(v)) \} \|_{\Gamma_I\cup\Gamma_D}^2
\end{align}
where
\begin{equation}
\| v \|_{a_h}^2 = a_h(v,v) \,, \qquad
\| v \|_{b_h}^2 = b_h(v,v) \,, \qquad
\| v \|_{s_h}^2 = s_h(v,v)
\end{equation}

\begin{lem}[Continuity and Coercivity]
The bilinear form $A_h$ is continuous on $V+V_h$, i.e., there exists a constant such that
\begin{equation} \label{eq:Ah-cont}
\left|A_h(v,w)\right| \lesssim \tn v \tn_h \tn w \tn_h \,, \qquad \forall v,w\in V + V_h
\end{equation}
For $\beta$ large enough $A_h$ is coercive on $V_h$, i.e., there exists a constant such that
\begin{equation} \label{eq:Ah-coer}
A_h(v,v) \gtrsim \tn v \tn_h^2 \,, \qquad \forall v \in V_h
\end{equation}
\end{lem}
\begin{proof}
\textbf{Continuity (\ref{eq:Ah-cont}).}
By the triangle and Cauchy--Schwarz inequalities we have
\begin{align}
| A_h(v,w) | &\leq |a_h(v,w)| + |s_h(v,w)| + \beta |b_h(v,w)| + |c_h(v,w)|
\\&\leq \| v \|_{a_h} \| w \|_{a_h}
+ \| v \|_{s_h} \| w \|_{s_h} + \beta \| v \|_{b_h}  \| w \|_{b_h}
\\&\quad \nonumber
+ \left| (\{\sigma(v)\cdot n\},[w])_{\Gamma_I \cup \Gamma_D} \right|
+ \left| (\{\sigma(w)\cdot n\},[v])_{\Gamma_I \cup \Gamma_D} \right|
\end{align}
where only the last two terms are not included in the energy norm. We make the split
\begin{equation} \label{eq:split-integral}
(\{\sigma(v)\cdot n\},[w])_{\Gamma_I \cup \Gamma_D}
=
\underbrace{(\{2\mu\epsilon(v)\cdot n\},[w])_{\Gamma_I \cup \Gamma_D}}_{I_\mu}
+
\underbrace{(\{\lambda\Trace(\epsilon(v))n\},[w])_{\Gamma_I \cup \Gamma_D}}_{I_\lambda}
\end{equation}
Rearranging the integrand and using the Cauchy--Schwarz inequality we obtain
\begin{align}
\left|
I_\mu
\right|
&=
\left|
(\{h^{-1}2\mu\}^{-1/2}\{2\mu\epsilon(v)\cdot n\},\{h^{-1}2\mu\}^{1/2}[w])_{\Gamma_I \cup \Gamma_D}
\right|
\\&\leq \label{eq:bound-Imu}
\left\|
\{h^{-1}2\mu\}^{-1/2}\{2\mu\epsilon(v)\cdot n\}
\right\|_{\Gamma_I \cup \Gamma_D}
\left\|
\{h^{-1}2\mu\}^{1/2}[w]
\right\|_{\Gamma_I \cup \Gamma_D}
\end{align}
and analogously for $I_\lambda$ we have
\begin{align}
\left|
I_\lambda
\right|
&=
\left|
(\{h^{-1}\lambda\}^{-1/2}\{\lambda\Trace(\epsilon(v))\},\{h^{-1}\lambda\}^{1/2}[w \cdot n])_{\Gamma_I \cup \Gamma_D}
\right|
\\&\leq \label{eq:bound-Ilambda}
\left\|
\{h^{-1}\lambda\}^{-1/2}\{\lambda\Trace(\epsilon(v))\}
\right\|_{\Gamma_I \cup \Gamma_D}
\left\|
\{h^{-1}\lambda\}^{1/2}[v\cdot n]
\right\|_{\Gamma_I \cup \Gamma_D}
\end{align}
which completes the proof of continuity as all remaining terms are now trivially bounded by the energy norm.

\paragraph{Coercivity (\ref{eq:Ah-coer}).}
The definition of $A_h$ gives
\begin{align}
A_h(v,v) &= \| v \|_{a_h}^2 + \| v \|_{s_h}^2 + \beta \| v \|_{b_h}^2
-2(\{\sigma(v)\cdot n\}, [v])_{\Gamma_I \cup \Gamma_D}
\\&\geq
\| v \|_{a_h}^2 + \| v \|_{s_h}^2 + \beta \| v \|_{b_h}^2
-2 \left| (\{\sigma(v)\cdot n\}, [v])_{\Gamma_I \cup \Gamma_D} \right|
\end{align}
By the split \eqref{eq:split-integral} with $w=v$, and the triangle inequality we for the last term have
\begin{equation}
\left|
(\{\sigma(v)\cdot n\},[v])_{\Gamma_I \cup \Gamma_D}
\right|
\leq 
\left| I_\mu \right| + \left| I_\lambda \right|
\end{equation}
Next using Young's inequality $2ab \leq \delta a^2 + \delta^{-1}b^2$ with $\delta>0$ on the bound for $\left|I_\mu\right|$ in \eqref{eq:bound-Imu} and likewise on the bound for $\left|I_\lambda\right|$ in \eqref{eq:bound-Ilambda} we obtain
\begin{align}
\left|
I_\mu
\right|
&\lesssim \label{eq:term1-Imu}
\delta
\left\|
\{h^{-1}2\mu\}^{-1/2}\{2\mu\epsilon(v)\cdot n\}
\right\|_{\Gamma_I \cup \Gamma_D}^2
+
\delta^{-1}
\left\|
\{h^{-1}2\mu\}^{1/2}[v]
\right\|_{\Gamma_I \cup \Gamma_D}^2
\\
\left|
I_\lambda
\right|
&\lesssim \label{eq:someeq04975}
\delta
\left\|
\{h^{-1}\lambda\}^{-1/2}\{\lambda\Trace(\epsilon(v))\}
\right\|_{\Gamma_I \cup \Gamma_D}^2
+
\delta^{-1}
\left\|
\{h^{-1}\lambda\}^{1/2}[v\cdot n]
\right\|_{\Gamma_I \cup \Gamma_D}^2
\end{align}
where we note that the $\delta^{-1}$ terms are included in $\|v\|_{b_h}^2$.
As a technical tool we now introduce the conjugate operator to the weighted average
\begin{equation}
\llbracket v \rrbracket = \frac{\mu_j}{\mu_i + \mu_j} v_i - \frac{\mu_i}{\mu_i + \mu_j} v_j
\quad\text{on $\partial\Omega_i\cap\partial\Omega_j$}
\,,\qquad
\llbracket v \rrbracket = v_i \quad\text{on $\partial\Omega_i\cap\Gamma_D$}
\end{equation}
which satisfies the following basic identity and inequalities
\begin{equation} \label{eq:new-weighted-id}
\{ a b \} = 2\{ a \}\langle b \rangle + \llbracket a \rrbracket [ b ] \, ,
\qquad
\frac{\llbracket c^{1/2} \rrbracket^2}{\{ c \}}
\leq
\frac{\{ c^{1/2} \}^2}{\{ c \}} \leq 1 \quad \text{for $c_i,c_j>0$}
\end{equation}
Using this identity we have
\begin{align}
\{2\chi\mu\epsilon(v)\cdot n\}
&=
2\{(h^{-1}2\mu)^{1/2}\}
\langle (h 2\mu)^{1/2}\epsilon(v)\cdot n \rangle
\\&\quad + \nonumber
\llbracket(h^{-1}2\mu)^{1/2}\rrbracket
[ (h 2\mu)^{1/2}\epsilon(v)\cdot n ]
\end{align}
which for the first term in \eqref{eq:term1-Imu} gives
\begin{align}
&\left\|
\{h^{-1}2\mu\}^{-1/2}\{2\mu\epsilon(v)\cdot n\}
\right\|_{\Gamma_I \cup \Gamma_D}^2 \nonumber
\\& \qquad\quad
\lesssim
\left\|
\frac{\{(h^{-1} 2\mu)^{1/2}\}}{\{h^{-1}2\mu\}^{1/2}}
\langle (h 2\mu)^{1/2}\epsilon(v)\cdot n \rangle
\right\|_{\Gamma_I \cup \Gamma_D}^2
\\& \qquad\qquad\ \nonumber
+
\left\|
\frac{\llbracket(h^{-1} 2\mu)^{1/2}\rrbracket}{\{h^{-1}2\mu\}^{1/2}}
\left[ (h 2\mu)^{1/2}\epsilon(v)\cdot n \right]
\right\|_{\Gamma_I \cup \Gamma_D}^2
\\& \qquad\quad \lesssim \label{eq:someeq85347}
\left\|
\langle (h 2\mu)^{1/2}\epsilon(v)\cdot n \rangle
\right\|_{\Gamma_I \cup \Gamma_D}^2
+
\left\|
\left[ (h 2\mu)^{1/2}\epsilon(v)\cdot n \right]
\right\|_{\Gamma_I \cup \Gamma_D}^2
\\&\qquad\quad
\lesssim
\sum_{i=0}^n
h_i (2_i \mu_i \epsilon(v_i),\epsilon(v_i))_{\partial\Omega_i}
\end{align}
where we in \eqref{eq:someeq85347} use inequality \eqref{eq:new-weighted-id} and in the last
inequality we use the triangle inequality on the jump and averages.
Equivalently for the first term in \eqref{eq:someeq04975}
we get
\begin{align}
\left\|
\{h^{-1}\lambda\}^{-1/2}\{\lambda\Trace(\epsilon(v))\}
\right\|_{\Gamma_I \cup \Gamma_D}^2
\lesssim 
\sum_{i=0}^n
h_i (\lambda_i \Trace(\epsilon(v_i))I,\epsilon(v_i))_{\partial\Omega_i}
\end{align}
so for the sum of these terms we have
\begin{multline}
\left\|
\{h^{-1}2\mu\}^{-1/2}\{2\chi\mu\epsilon(v)\cdot n\}
\right\|_{\Gamma_I \cup \Gamma_D}^2
+
\left\|
\{h^{-1}\lambda\}^{-1/2}\{\chi\lambda\Trace(\epsilon(v))\}
\right\|_{\Gamma_I \cup \Gamma_D}^2
\\
\lesssim
\sum_{i=0}^n h_i (\sigma_i(v_i),\epsilon(v_i))_{\partial\Omega_i}
\lesssim
\sum_{i=0}^n
(\sigma_i(v_i),\epsilon(v_i))_{\Omega_i} + s_{i,h}(v_i,v_i)
\end{multline}
where we finally utilize the inverse inequality \eqref{eq:inverse-subdomain}.
In summary this calculation yields
\begin{equation}
\left|
(\{\sigma(v)\cdot n\},[v])_{\Gamma_I \cup \Gamma_D}
\right|
\lesssim
\delta \left(
\| v \|_{a_h}^2 + \| v \|_{s_h}^2
\right)
+
\delta^{-1} \| v \|_{b_h}^2
\end{equation}
and thus, choosing $\delta$ small enough such that we can hide the $\delta$-term and $\beta$ large enough such that we can hide the $\delta^{-1}$-term in the penalty term will produce a coercive method.
\end{proof}

\begin{rem}\label{rem:muzero} Note that in this proof of coercivity it is not central that the lower bound $\mu_\mathrm{min} \leq \mu$ holds everywhere in $\Omega$ but rather that the inverse inequality \eqref{eq:inverse-subdomain} holds on each subdomain.
Through the selection of $s_{i,h}$ the method actually can be made to accommodate extremely small values for $\mu_\mathrm{min}$, even zero.
For example, if we would choose $s_{0,h}$
as in \eqref{eq:stab-form} with the modification that we add the stabilization on every interior face in $\mcT_{0,h}$ we could allow $\chi=0$ (and in turn $\mu=0$) anywhere in $\Omega_0$ and the inverse inequality would still hold as long as $\{\mu\} \neq 0$ on a non-empty part of $\partial\Omega_0$. 
\end{rem}

\begin{rem}
For results on  the existence and uniqueness of a discrete solution, condition number estimates and a priori error estimates, we refer to \cite{Burman2012}.
\end{rem}

\section{Topology Optimization} \label{section:top-opt}

\paragraph{A Schematic View of an Optimization Procedure.}
For pedagogical purposes we first outline a simple steepest descent type algorithm to give a general view of the steps taken and the necessary quantities to compute in the optimization procedure. For $\chi^\ell \in W$, fulfilling the volume constraint \eqref{eq:vol-constraint}, we formulate the steepest descent step
\begin{align} \label{eq:schematic-descent-step}
\chi^{\ell+1} &= \chi^{\ell} - \alpha^{\ell} \max_{\delta\chi\in W , \, \langle\delta\chi,\delta\chi\rangle=1} \langle d_\chi J, \delta\chi \rangle
\end{align}
where $\alpha^\ell$ is a step size and the maximum gives the descent direction. By iterating \eqref{eq:schematic-descent-step} with a small enough step size $\alpha^\ell$ we expect to converge to a minima of $J(\chi)$, i.e., a solution to the minimization problem \eqref{eq:min}.
Evaluating the total derivative is somewhat intricate as $J(\chi)=J_\chi(u_\chi)$ where $u_\chi$ is the solution to the elasticity problem \eqref{eq:weak}. However, using \eqref{eq:total-to-partial}, which holds
under a certain constraint discussed below, we can replace this total derivative with a partial derivative of the Lagrangian \eqref{eq:lagrangian}
\begin{equation} \label{eq:simplified-part-der}
\langle d_\chi J, \delta\chi \rangle
=
\langle \partial_\chi \mcL, \delta\chi \rangle
\end{equation}
Focusing on the case of compliance, i.e., $J(\chi)=l_\chi(u_\chi)$, the Lagrangian may be expressed
\begin{equation}
\mcL(\chi,u,\eta) = \Pi_\chi(u) + \eta \Lambda_\chi
\end{equation}
where we recall from \eqref{eq:compliance-lagrangian} the Ritz and volume constraint functionals
\begin{align}
\Pi_\chi(u) &= -2\left( \frac{1}{2}a_\chi(u,u) - l_\chi(u)\right)
\quad\text{and}\quad
\Lambda_\chi = ( \chi - \Delta_{\mathrm{vol}} , 1 )_{\Omega_0}
\end{align}
For a discretized $\chi$ described via parameters $\chi|_k$ we can formulate the descent direction using the usual gradient, which yields the steepest descent step
\begin{equation}
\chi^{\ell+1} = \chi^{\ell} - \alpha^{\ell} \nabla_{\chi} \mcL
\qquad\text{where}\quad
(\nabla_{\chi} \mcL) |_k = \frac{\partial\mcL}{\partial\chi|_k}
\end{equation}
While such an iteration seems simple enough, the numerous constraints, implied by the derivation and also by properties we desire for the final design to be useful in practice, make it non-trivial to construct an algorithm for the optimization procedure. The key considerations are:
\begin{itemize}
\item \emph{Density Field Contrast.}
To be able to get a clear view of where in the design domain material should be present we want the resulting density field $\chi$ to be of high contrast with as little intermediate values as possible.
This is handled by posing the optimization procedure in terms of an auxiliary field $\rho$ on which the density field $\chi$ depends in a non-linear way.
We describe this in the paragraph ``Density Field'' below.

\item \emph{Mesh Size Independence.}
The results should be independent of mesh size, in the sense that a refined computational grid should not give a drastically different result. This is solved by using a filter on either the density field or the sensitivities. A filter also gives the designer some control of the resulting density field.
The sensitivity filter we employ in our numerical examples is adapted to the cut finite element method and is described in the paragraph ``Sensitivity Filter'' below.

\item \emph{Update Constraints.}
For the above derivation to hold there are a number of constraints that must be fulfilled, which complicates the scheme. In particular, \eqref{eq:simplified-part-der} only holds at points $(\chi,u,p,\eta)$ where: $\chi$ satisfies the volume constraint \eqref{eq:vol-constraint}; $u$ satisfies the primal problem \eqref{eq:primal-prob}; $p$ satisfies the dual problem \eqref{eq:dual-prob}; and $\eta$ is chosen such that the partial derivative on the right of \eqref{eq:simplified-part-der} is tangent to the manifold of admissible density functions, i.e., density functions satisfying the volume constraint.
This last requirement implies that the update scheme must ensure that the next iteration of the density field $\chi^{\ell+1}$ also satisfies the volume constraint  \eqref{eq:vol-constraint}. A heuristic procedure for iterating the density while maintaining the constraints is detailed in the paragraph ``Updating Scheme'' below.
\end{itemize}

\paragraph{Density Field.}
To obtain an optimal density field $\chi$ with high contrast
we adapt the SIMP topology optimization procedure, see \cite{Bendsoe1989,Zhou1991}, to cut finite element methods.
On the design domain $\Omega_0$ we introduce the auxiliary field
\begin{equation} \label{eq:approx-rho}
\rho_0 \in \{ w \in W_{0,h,k}|_{\Omega_0} \, : \, 0 \leq w \leq 1 \text{ on $\Omega_0$} \}
\end{equation}
and as an approximation for the density on $\Omega_0$ we choose
\begin{equation} \label{eq:approx-chi}
\chi \approx \chi_\mathrm{min} + \rho_0^q (1 - \chi_\mathrm{min})
\end{equation}
By construction $\chi : \Omega_0 \to [\chi_\mathrm{min},1]$ and
raising the power $q \geq 1$ will increase the penalization of intermediate values of $\rho_0$ and thus produce a density field $\chi$ with sharper transitions between regions in $\Omega_0$ with and without material. As noted in Remark~\ref{rem:muzero} the cut finite element method is not inherently sensitive to choosing the value of $\chi_\mathrm{min}$ large enough and the method can actually be devised to handle even the case $\chi_\mathrm{min}=0$.
Recall that the density is typically represented on a finer grid than the finite element solution as the underlying mesh in $W_{0,h,k}$ is a uniform $k$-refinement of the mesh in $V_{0,h}$.
The effect of the choice of refinement in the density mesh in relation to the polynomial order $p$ in Lagrange-type finite element spaces was numerically studied in \cite{MR3670454}, where a similar multi-resolution topology optimization approach was used.

\paragraph{Sensitivity Filter.}
It is well known that the standard SIMP procedure suffers from checker board patterns in the density field, especially when using low order elements, see \cite{Li2001} and the references therein. The compliance problem itself is also not well posed, yielding numerical approximations consisting of finer and finer structures when refining the discretization grid. Both these issues are commonly remedied by introducing a filter, applied to either the density field or to the sensitives. In CutFEM the domain $\Omega_0$ may cut the mesh in an arbitrary fashion and therefore the element volume $|T \cap \Omega_0|$ typically varies even on structured meshes, and may in fact approach zero. We devise a sensitivity filter that takes into account both the variable element sizes and the interface conditions as follows:
\begin{itemize}
\item 
Analogously to the auxiliary field $\rho_0$ in the design domain $\Omega_0$, we in the nondesign domain regions $\{\Omega_i\}_{i=1}^n$ define fields $\{\rho_i\}_{i=1}^n$ on the parametric meshes refined such that their physical element sizes are approximately the same as for $\rho_0$. 
While these fields clearly are constant $\rho_i=1$, $i=1,\dots,n$, they will be used in the construction of a filter with suitable behavior near the interfaces.

\item We number the elements in the mesh used to describe $\rho_i$ from 1 to $N_i$ and denote element $j$ in the mesh by $T_{i,j}$.
In each element $T_{i,j}$ the auxiliary field is constant and
we employ the shorthand notation
\begin{equation}
\rho_{i,j} = \rho_i |_{T_{i,j}}
\end{equation}
and for the element volume we write
\begin{equation}
V_{i,j} = | T_{i,j} \cap \Omega_i |
\end{equation}

\item
We define the discrete weight factor
  \begin{equation} \label{eq:weight-factor}
    H_{i,j}^k = \max\left(0,R_\mathrm{min} - \mathrm{dist}(T_{0,k},T_{i,j}) \right)
  \end{equation}
  where $R_\mathrm{min}$ is the filter radius and $\mathrm{dist}(T_{0,k},T_{i,j})$ is the distance between the centroids of elements $T_{0,k}$ and $T_{i,j}$.
  
\item
Inspired by the sensitivity filter employed in \cite{Sigmund2007} for meshes with varying element sizes
our sensitivity filter in $\Omega_0$ takes the form
\begin{equation}\label{eq:filter}
\frac{\widetilde{\partial \Pi_\chi}}{\partial \rho_{0,k}}
= 
\frac{V_{0,k}}{\max(\rho_{0,k}\, ,\gamma)} \frac{\sum\limits_{i=0}^n \sum\limits_{j=1}^{N_i} H^k_{i,j} \rho_{i,j}  \frac{\partial \Pi_\chi}{\partial \rho_{i,j}} / \max(V_{i,j} ,\gamma)  }{\sum\limits_{i=0}^n \sum\limits_{j=1}^{N_i} H^k_{i,j}}
\end{equation}
where $\gamma > 0$ is some very small parameter used in combination with a maximum for avoiding numerical issues when the denominators tend to zero.
Note that as the derivative $\frac{\partial \Pi_\chi}{\partial \rho_{i,j}}$ is an integral over
$T_{i,j} \cap \Omega_i$ and $V_{i,j}=|T_{i,j} \cap \Omega_i|$ the quotient $\frac{\partial \Pi_\chi}{\partial \rho_{i,j}} / V_{i,j}$ is actually a dimensionless quantity.

\end{itemize}

\begin{rem}[Ghost Derivatives]
Our sensitivity filter \eqref{eq:filter} contains derivatives $\frac{\partial \Pi_\chi}{\rho_{i,j}}$ also in nondesign domain regions $\Omega\setminus\Omega_0$, which might seem strange as $\rho_{i}$ is fixed in those parts, see Figure~\ref{fig:filter}. The reason for including these `ghost derivatives' in the filter is to avoid the removal of thin layers of material close to interfaces. Thus, the filter \eqref{eq:filter} makes no distinction between the various parts of $\Omega$.
\end{rem}

\begin{rem}[Choice of Filter]
While the choice of filter is obviously an important topic (see, e.g., \cite{Bourdin2001,Sigmund2007,MR3620802}), we do not view our particular choice of filter  \eqref{eq:filter} to be central for employing CutFEM based topology optimization. Most likely a density or PDE based filter could just as easily be adapted to the CutFEM situation and yield similar results. In practical applications the evaluation speed of the filter is crucial. For that reason it would be interesting to explore the possibility to devise a filter which utilizes structured background grids for speed while still allowing design domains with complicated geometries via CutFEM.
\end{rem}

\begin{figure}\centering
\includegraphics[width=0.3\linewidth]{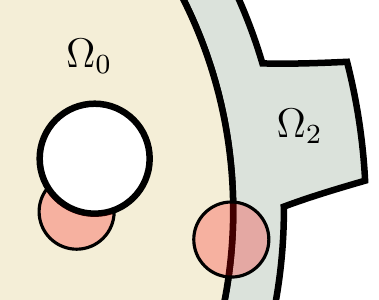}
\caption{Illustration of sensitivity filter area at two positions in $\Omega_0$ within the domain given in Figure~\ref{fig:domain}; on the boundary and on the interface. Note that the filter in the vicinity of the interface also takes into account parts of the domain outside the design domain $\Omega_0$.}
\label{fig:filter}
\end{figure}

\paragraph{Updating Scheme.} 
To update the design domain auxiliary field $\rho_0$ we use the optimal criteria method following a heuristic updating scheme, see \cite{Bendsoe95, Andreassen2011}.
The procedure is as follows:
\begin{itemize}
\item
For a given auxiliary field $\rho_0^\ell$, such that the density field $\chi^\ell$ satisfies the volume constraint \eqref{eq:vol-constraint}, we solve the elasticity problem \eqref{eq:weak}. The solution $u_{\chi^\ell}$ is used for evaluating derivatives of $\Pi_\chi$.

\item
The auxiliary field $\rho_0^\ell$ is elementwise updated according to the heuristic scheme
\begin{equation}\label{eq:optimal}
\rho^{\ell+1}_{0,k} =
\begin{cases}
\max(0, \rho_{0,k}^\ell - m)&\mathrm{if}\  \rho_{0,k}^\ell B_{0,k} \leq \max(0, \rho_{0,k}^\ell - m)
\\
\min(1, \rho_{0,k}^\ell + m)&\mathrm{if}\  \rho_{0,k}^\ell B_{0,k} \geq \min(1, \rho_{0,k}^\ell + m)
\\
\rho_{0,k}^\ell B_{0,k} \quad&\text{otherwise}
\end{cases}
\end{equation}
where $m \in (0,1]$ is a user specified positive move limit and
\begin{equation}
B_{0,k} = \frac{1}{\eta} \frac{\widetilde{\partial \Pi_\chi}}{\partial\rho_{0,k}} \Big/ \frac{{\partial \Lambda_\chi}}{\partial\rho_{0,k}}
\end{equation}

The unknown Lagrange multiplier $\eta$ is found through a simple bisection procedure
where $\rho_0$ is repeatedly updated using \eqref{eq:optimal} with different values of $\eta$ until the volume constraint \eqref{eq:vol-constraint} is satisfied.

\item
Instead of finding an initial guess for $\rho_0$ for which the volume constraint \eqref{eq:vol-constraint} is satisfied we first assume material everywhere in $\Omega_0$ and then gradually enforce the volume constraint during the iteration procedure by exchanging ${\Delta}_\mathrm{vol}$ in \eqref{eq:vol-constraint} by
\begin{equation}\label{eq:Delta_vol}
    \widehat{\Delta}_\mathrm{vol} = {\Delta}_\mathrm{vol} + (1-{\Delta}_\mathrm{vol})\max(0,1-i/\mathrm{RAMP})
\end{equation}
where $i$ denotes the iteration number
and we ramp up the volume constraint to the desired value while $i<\mathrm{RAMP}$.

\end{itemize}

\section{Numerical Results}

To illustrate the density based cut finite element topology optimization including parametric nondesign domain regions we in this section present some numerical experiments in 2D.

\subsection{Experimental Set-up}

\paragraph{Quadrature.}
The integration procedure used was outlined in \cite{JonLarLar17} and does not differentiate between cut or non-cut elements. Note that integration in the design domain is based on the $k$-refined mesh \eqref{eq:k-ref-mesh} such that the density is constant on each element.

\paragraph{Approximation Spaces.}
We set up the approximation spaces as follows:
\begin{itemize}
\item In every subdomain $\Omega_i \subset \Omega$ the underlying finite element space is constructed as tensor product B-splines of maximum regularity with polynomial order $p=2$. Thus, the approximation spaces will be $C^1$ within each subdomain. 

\item In the design domain $\Omega_0$ the mesh is allowed to be cut by the geometry.
By rotating the background mesh counter clockwise $\pi/7$ radians
we avoid using a mesh which utilize the general structure of the geometry, which might yield too optimistic results. This produces a variety of cut elements all along the design domain boundary and also breaks symmetry.

\item The auxiliary field $\rho_0$ describing the density is discretized according to \eqref{eq:approx-rho}, i.e., $\rho_0\in W_{0,h,k}|_{\Omega_0}$ is piecewise constant on a mesh which is constructed as $k$ uniform refinements of the finite element mesh.
In the experiments we use one uniform refinement ($k=1$) and as an initial guess we take $\rho_0=1$.

\item In the nondesign domain regions the meshes are parametrically mapped to the geometries using smooth mapping such that the boundaries fit the mapped meshes perfectly. For the curved parts we use biquadratic mappings and for the ring we use a polar mapping.
 
\item Note that there is no need to use modified B-spline basis functions near the boundary even on the parametrically fitted nondesign domain regions as the Dirichlet boundary and interface conditions are weakly imposed using Nitsche's method.

\end{itemize}

\paragraph{Parameter Values.} The experiments share the following parameter values:
\begin{itemize}
\item
We assume constant material properties throughout in all parts of the domain and we use a $E$-modulus $E=1$ and a Poisson's ratio $\nu=0.3$.
From these material parameters we find the Lamé parameters via the relationships
\begin{equation}
\widehat{\mu_i}=\frac{E}{2(1+\nu)}
\,,\qquad
\widehat{\lambda_i}=\frac{E\nu}{(1+\nu)(1-2\nu)}
\end{equation}
and recall that $\mu_i = \chi\widehat{\mu}_i$ and $\lambda_i = \chi\widehat{\lambda}_i$.

\item
We use a Nitsche's penalty parameter $\beta=10 \times p^2$ and for the CutFEM in the design domain $\Omega_0$ we use the ghost penalty stabilization form $s_{0,h}$ defined in \eqref{eq:stab-form} with parameters
\begin{equation}
\gamma_{0,k}= \widehat{\mu}_0 \frac{10^{-4}}{k!}
\end{equation}

\item
For the optimization procedure we set the desired proportion of material in the design domain $\Omega_0$ to $\Delta_\mathrm{vol}=0.5$. This volume constraint is gradually enforced during the $\mathrm{RAMP}=100$ first iterations. As to penalize intermediate values in the approximation of the density $\chi$ we select the power $q=3$ in \eqref{eq:approx-chi}.
For the sensitivity filter we choose a filter radius of $R_{\mathrm{min}}=1.2 \times h$ in \eqref{eq:weight-factor} and take $\gamma=10^{-6}$ in \eqref{eq:filter}. For the updating scheme we employ a move limit of $m=0.2$ in \eqref{eq:optimal}.
\end{itemize}

\subsection{Numerical Experiments}

In the following two examples we consider cantilevers where we optimize the geometry of the design domain $\Omega_0$ with respect to compliance, i.e., we seek the critical point characterized by \eqref{eq:compliance-lagrangian}.

\paragraph{Cantilever with Parametric Reinforcements.}
As a first example we consider the design domain and nondesign regions illustrated in Figure~\ref{fig:parametric-beam}. The leftmost boundary (red lines in Figure~\ref{fig:parametric-beam-mesh}) is the Dirichlet boundary and we apply a downward facing traction force on the inner boundary of the ring (blue circle in Figure~\ref{fig:parametric-beam-mesh}). The magnitude of the traction force varies horizontally as a parabolic function that is zero on the leftmost and rightmost points of the circle. Thus, the geometry is symmetric along the horizontal midline and the boundary conditions on both sides of the midline are the same. 
Changing the sign of the traction force would simply mirror the problem about the midline. This, in combination with linear elasticity, leads us to expect the optimized geometry to be symmetric about the midline.
It can be seen in the final geometry in Figure~\ref{fig:parametric-beam-final} that this is also the case even though we have rotated the background grid in the design domain $\Omega_0$ to break symmetry. Interestingly, the optimized geometry mirrors the curved beams in the nondesign domain regions, which we suppose is a way of approximating a straight structure which does not twist when pulled. Looking at Figure~\ref{fig:parametric-beam-stress} we see that the method manages to produce stresses without discontinuities over the interfaces.

In Figure~\ref{fig:parametric-beam-mod} we modify this example by introducing a nondesign domain region in a part of the design domain $\Omega_0$ where the previous optimized geometry had material to see if we end up with the same geometry as before. It should be noted that we do not make any correction to $\Delta_\textrm{vol} = 0.5$ and therefore the final geometry will have a slightly larger proportion of material than before.
As seen in Figure~\ref{fig:parametric-beam-final-mod} the final geometry is very similar to the case without the modification.

\paragraph{Cantilever with Truss Structure Reinforcements.}
In the second example we consider the design domain and nondesign domain regions illustrated in Figure~\ref{fig:truss-structure}. Here the nondesign domain regions constitute a frame and a number of internal beams in a truss structure like arrangement.
The leftmost boundary (red line in Figure~\ref{fig:truss-structure-mesh}) is the Dirichlet boundary and we apply a unitary downward facing traction force on the rightmost boundary (blue line in Figure~\ref{fig:truss-structure-mesh}).
Note that the internal beams are not placed optimally, i.e., they are not placed in positions where the optimization procedure would necessarily place material. In Figure~\ref{fig:truss-structure-final} we see that the final design actually utilizes most of the internal structures given by the nondesign domain regions.

\def\colRatioOne{0.45}
\def\colRatioTwo{0.9}
\begin{figure}\centering
\begin{subfigure}[b]{\colRatioOne\textwidth}\centering
\includegraphics[width=\colRatioTwo\linewidth]{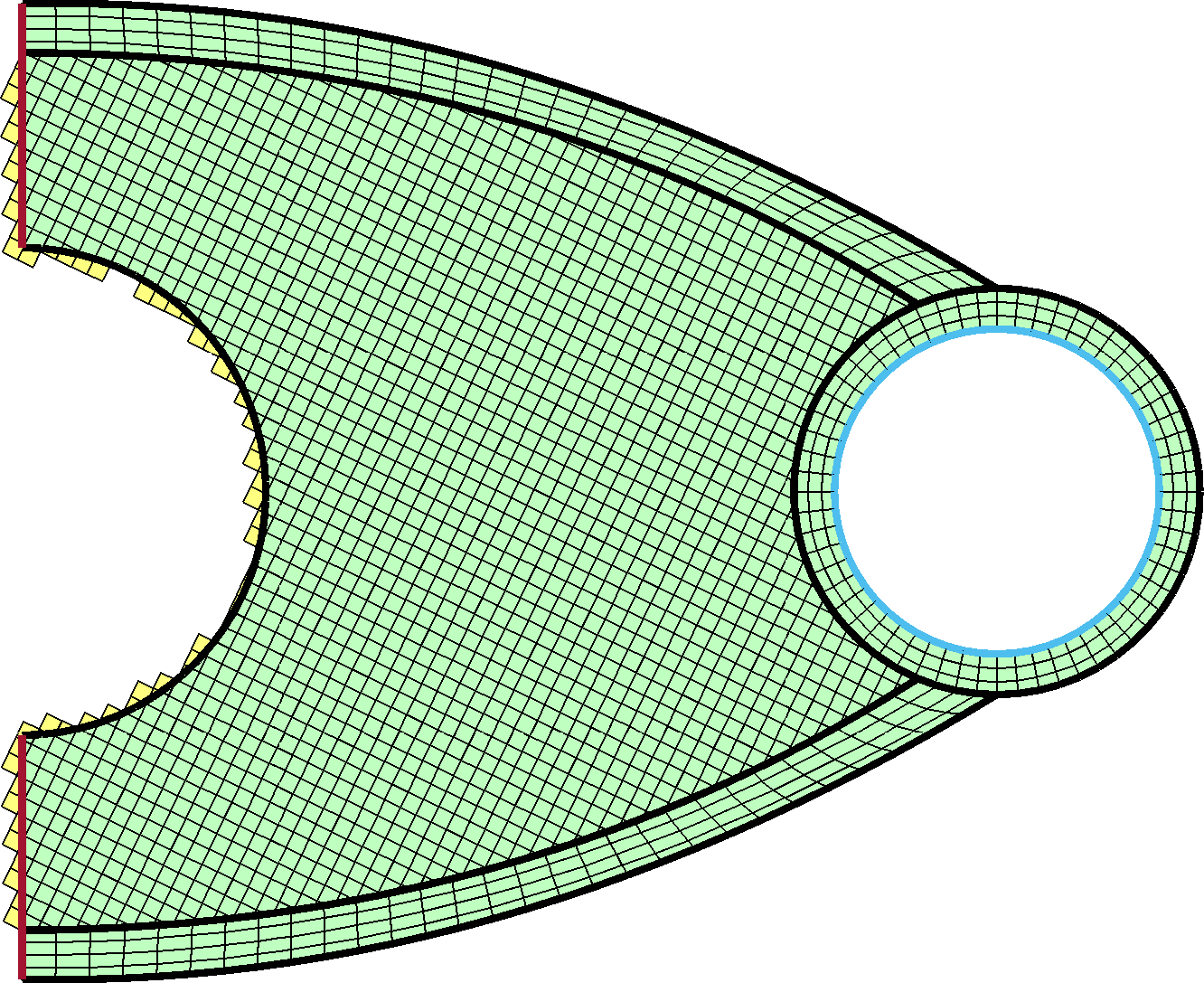}
\caption{Finite element meshes}
\label{fig:parametric-beam-mesh}
\end{subfigure}
\begin{subfigure}[b]{\colRatioOne\textwidth}\centering
\includegraphics[width=\colRatioTwo\linewidth]{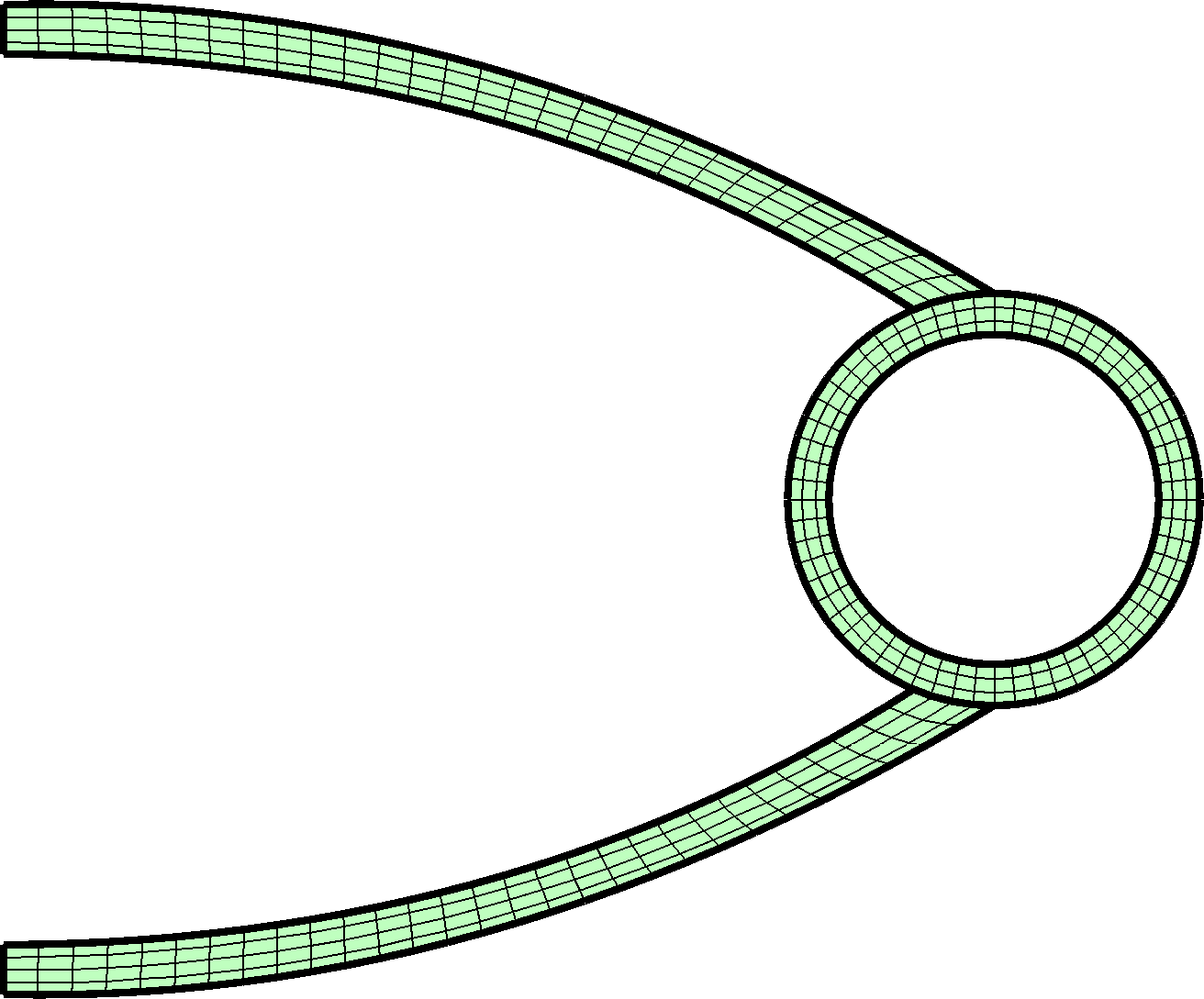}
\caption{Nondesign domain regions $\{\Omega_i\}_{i=1}^N$}
\label{fig:parametric-beam-given}
\end{subfigure}

\begin{subfigure}[b]{\colRatioOne\textwidth}\centering
\includegraphics[width=\colRatioTwo\linewidth]{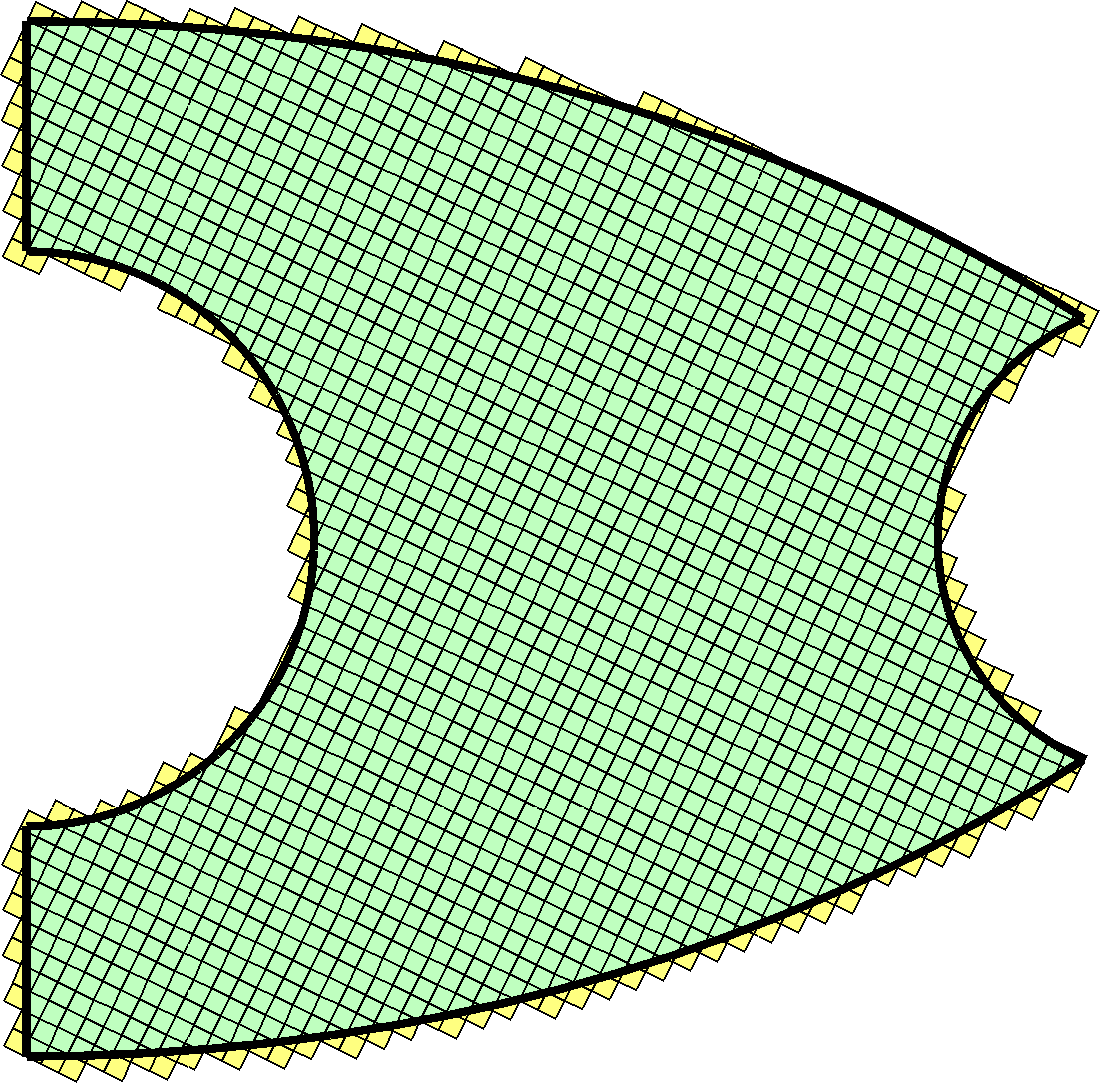}
\caption{Design domain $\Omega_0$}
\label{fig:parametric-beam-design}
\end{subfigure}
\begin{subfigure}[b]{\colRatioOne\textwidth}\centering
\includegraphics[width=\colRatioTwo\linewidth]{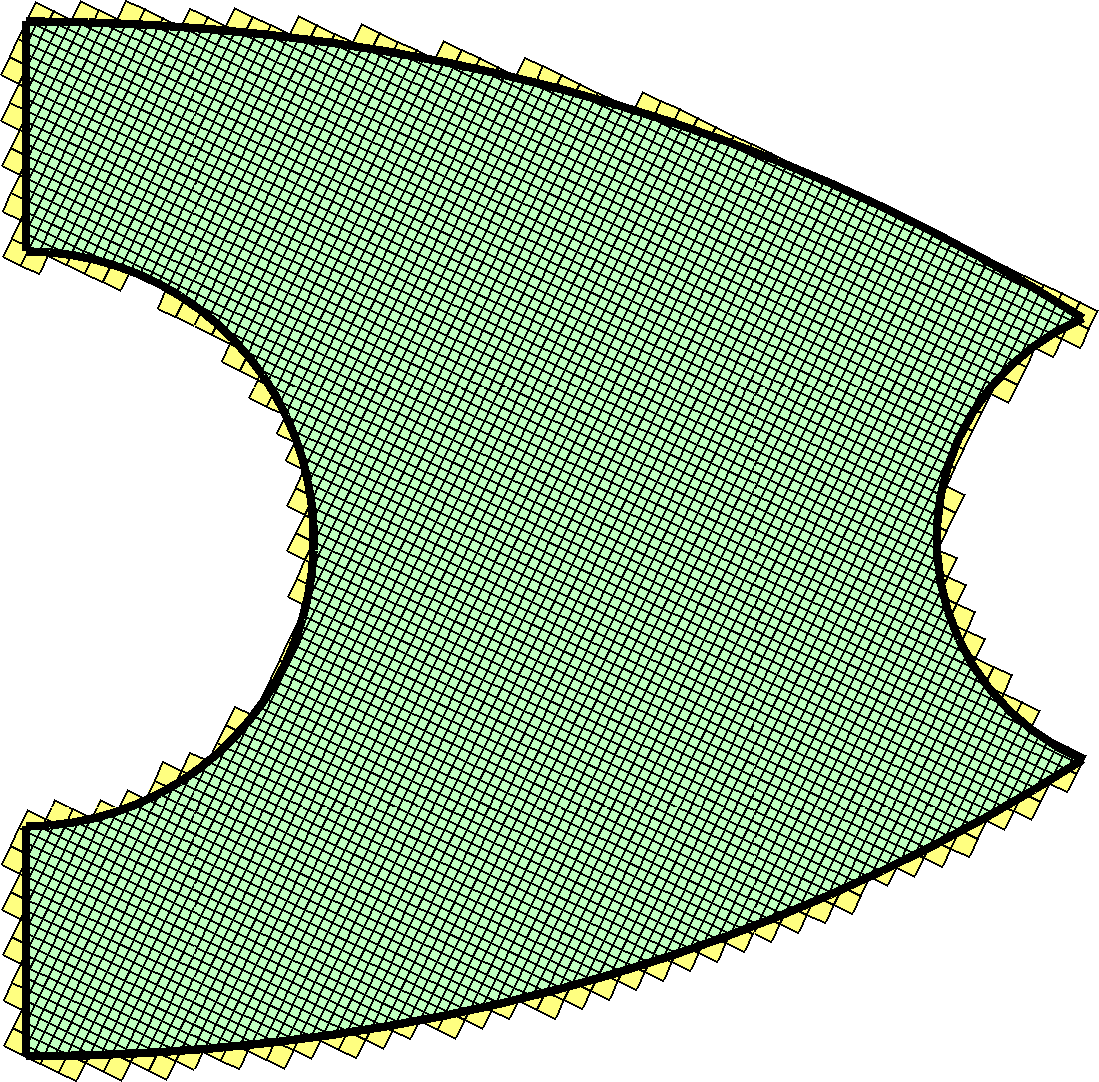}
\caption{Density resolution}
\label{fig:parametric-beam-density}
\end{subfigure}

\begin{subfigure}[b]{\colRatioOne\textwidth}\centering
\includegraphics[width=\colRatioTwo\linewidth]{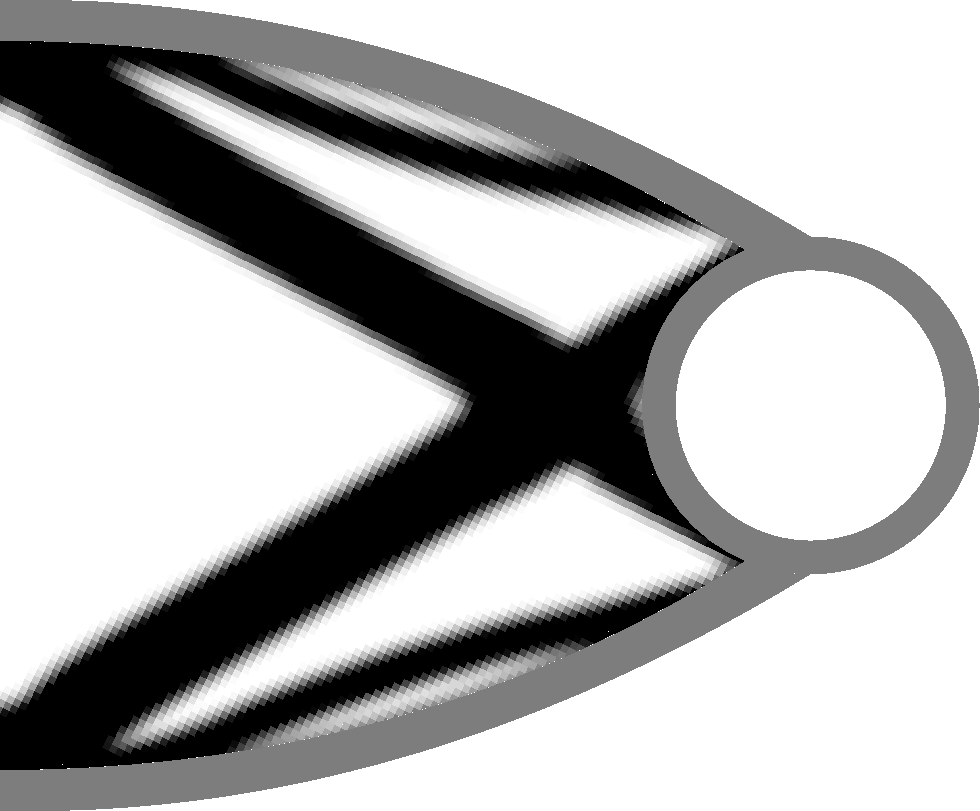}
\caption{Final geometry}
\label{fig:parametric-beam-final}
\end{subfigure}
\begin{subfigure}[b]{\colRatioOne\textwidth}\centering
\includegraphics[width=\colRatioTwo\linewidth]{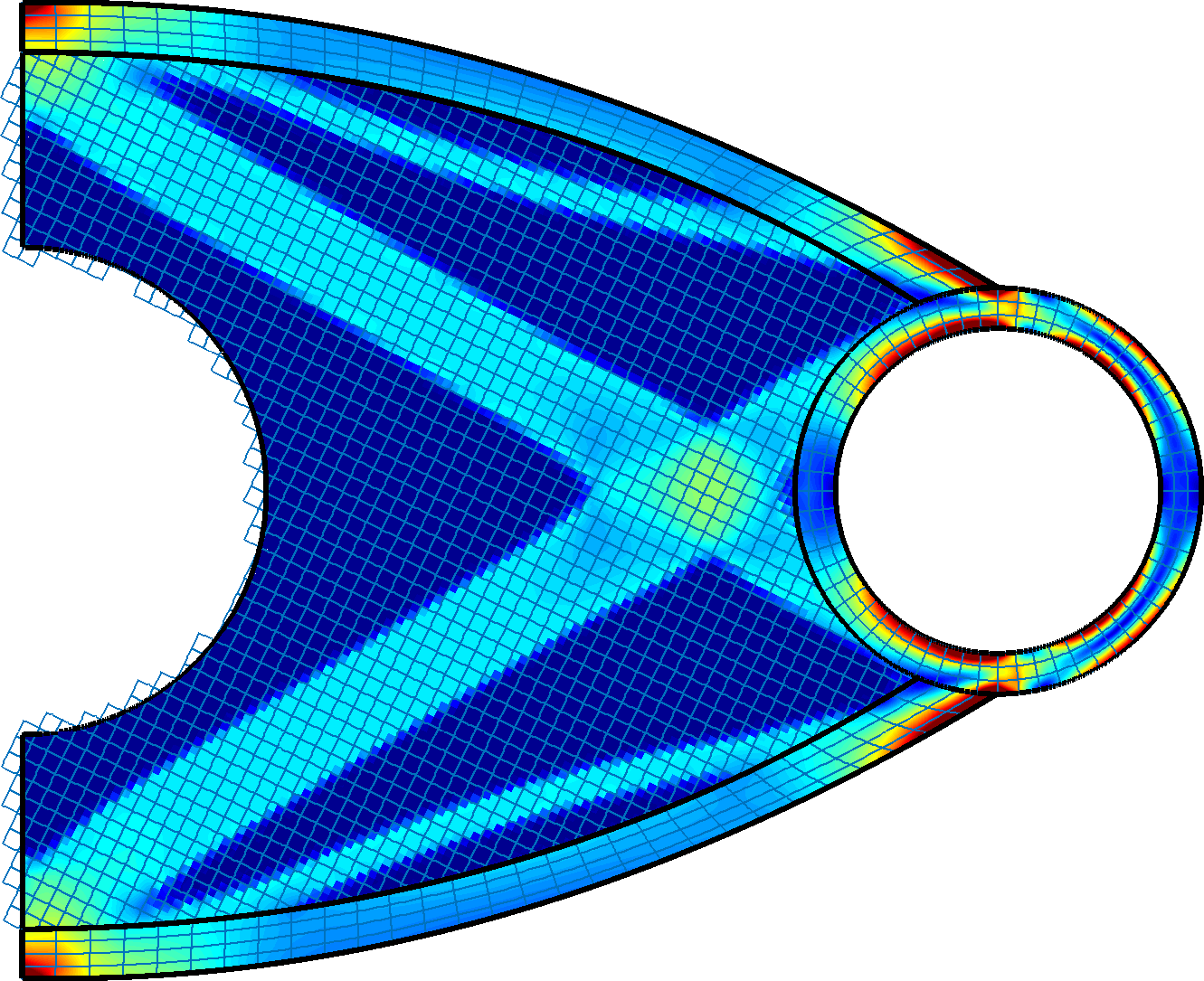}
\caption{Von-Mises stress in final geometry}
\label{fig:parametric-beam-stress}
\end{subfigure}

\caption{
{Cantilever with parametric reinforcements.}
}
\label{fig:parametric-beam}
\end{figure}

\def\colRatioOne{0.45}
\def\colRatioTwo{0.9}
\begin{figure}\centering
\begin{subfigure}[b]{\colRatioOne\textwidth}\centering
\includegraphics[width=\colRatioTwo\linewidth]{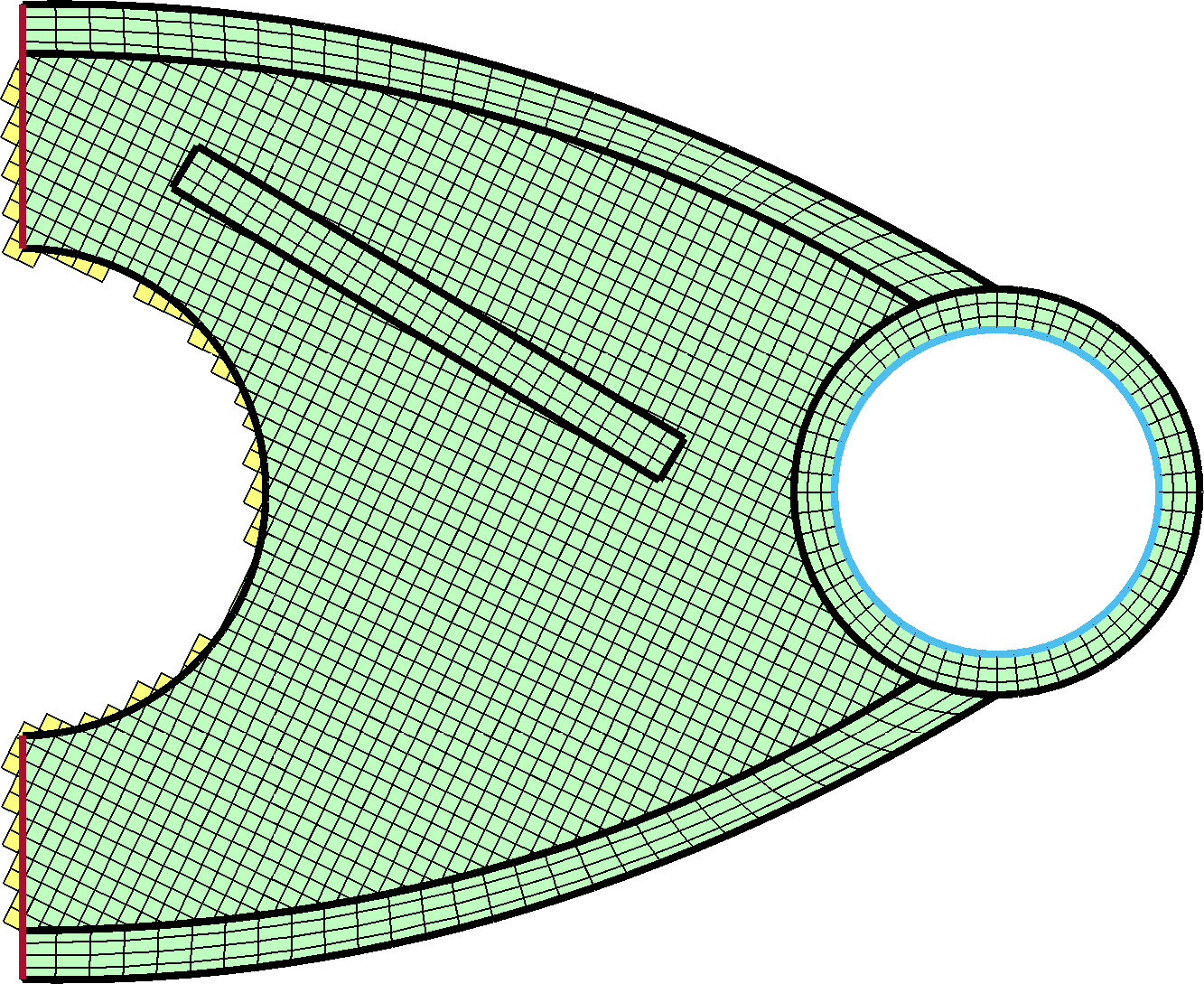}
\caption{Finite element meshes}
\label{fig:parametric-beam-mesh-mod}
\end{subfigure}
\begin{subfigure}[b]{\colRatioOne\textwidth}\centering
\includegraphics[width=\colRatioTwo\linewidth]{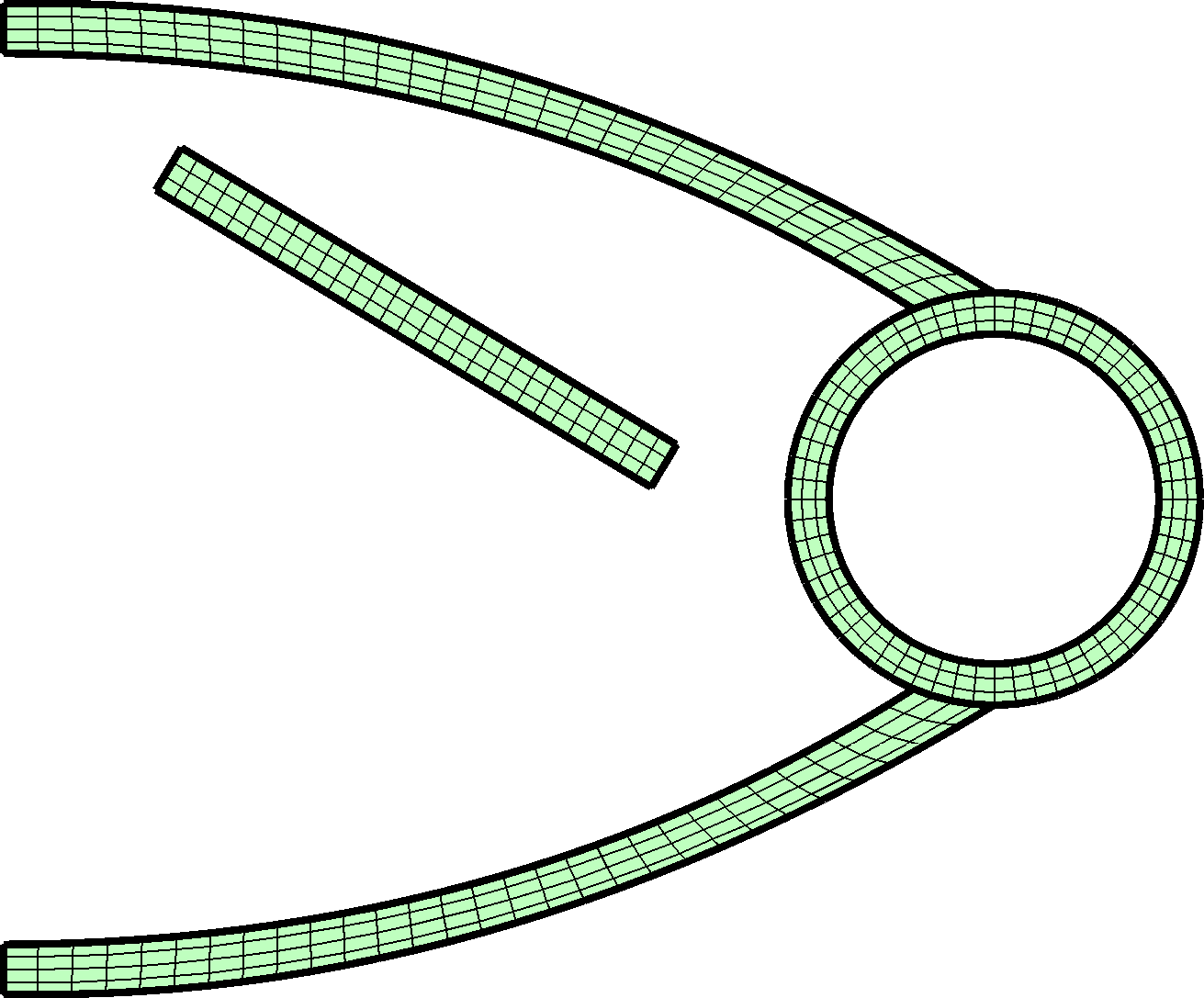}
\caption{Nondesign domain regions $\{\Omega_i\}_{i=1}^N$}
\label{fig:parametric-beam-given-mod}
\end{subfigure}

\begin{subfigure}[b]{\colRatioOne\textwidth}\centering
\includegraphics[width=\colRatioTwo\linewidth]{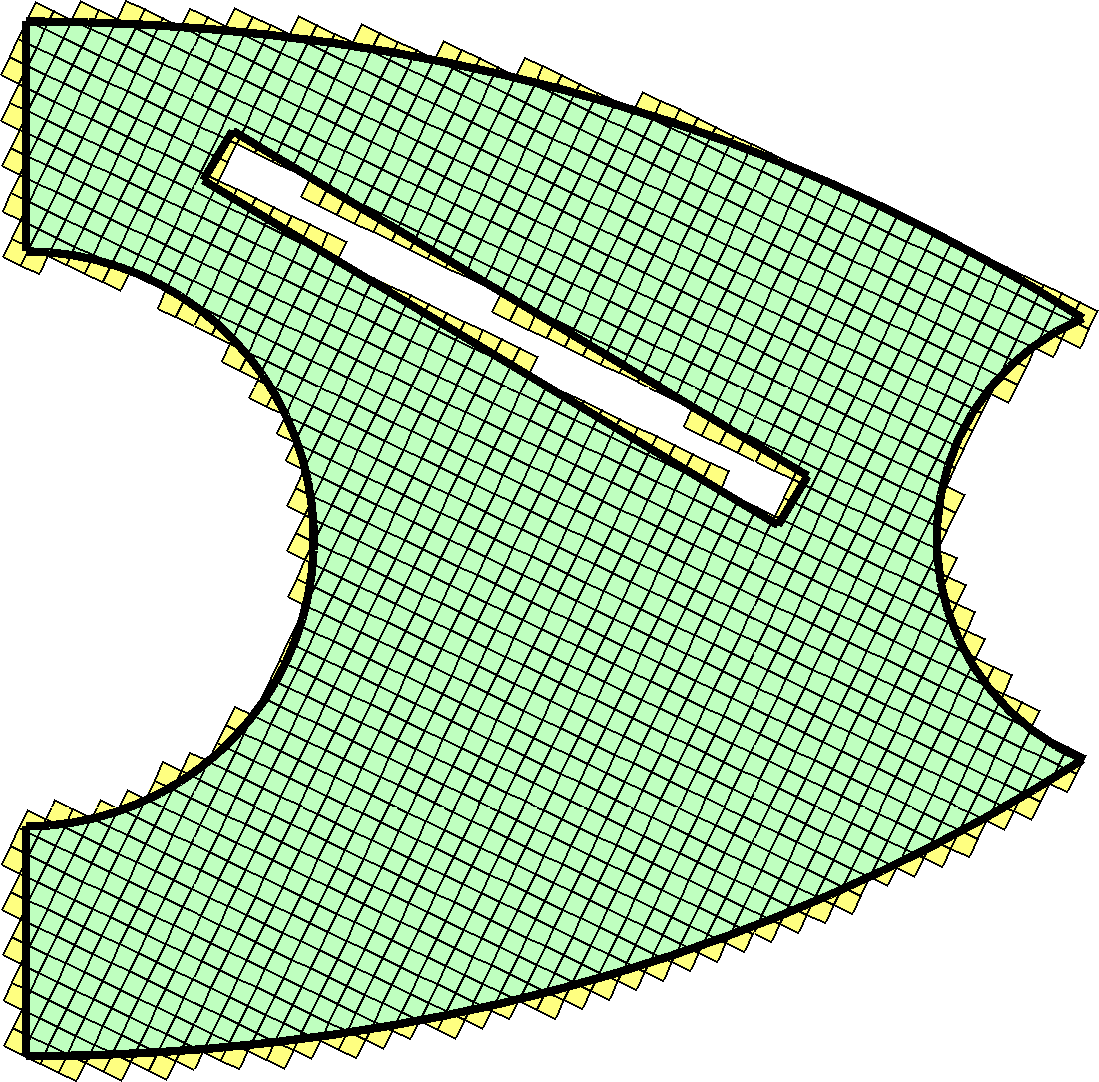}
\caption{Design domain $\Omega_0$}
\label{fig:parametric-beam-design-mod}
\end{subfigure}
\begin{subfigure}[b]{\colRatioOne\textwidth}\centering
\includegraphics[width=\colRatioTwo\linewidth]{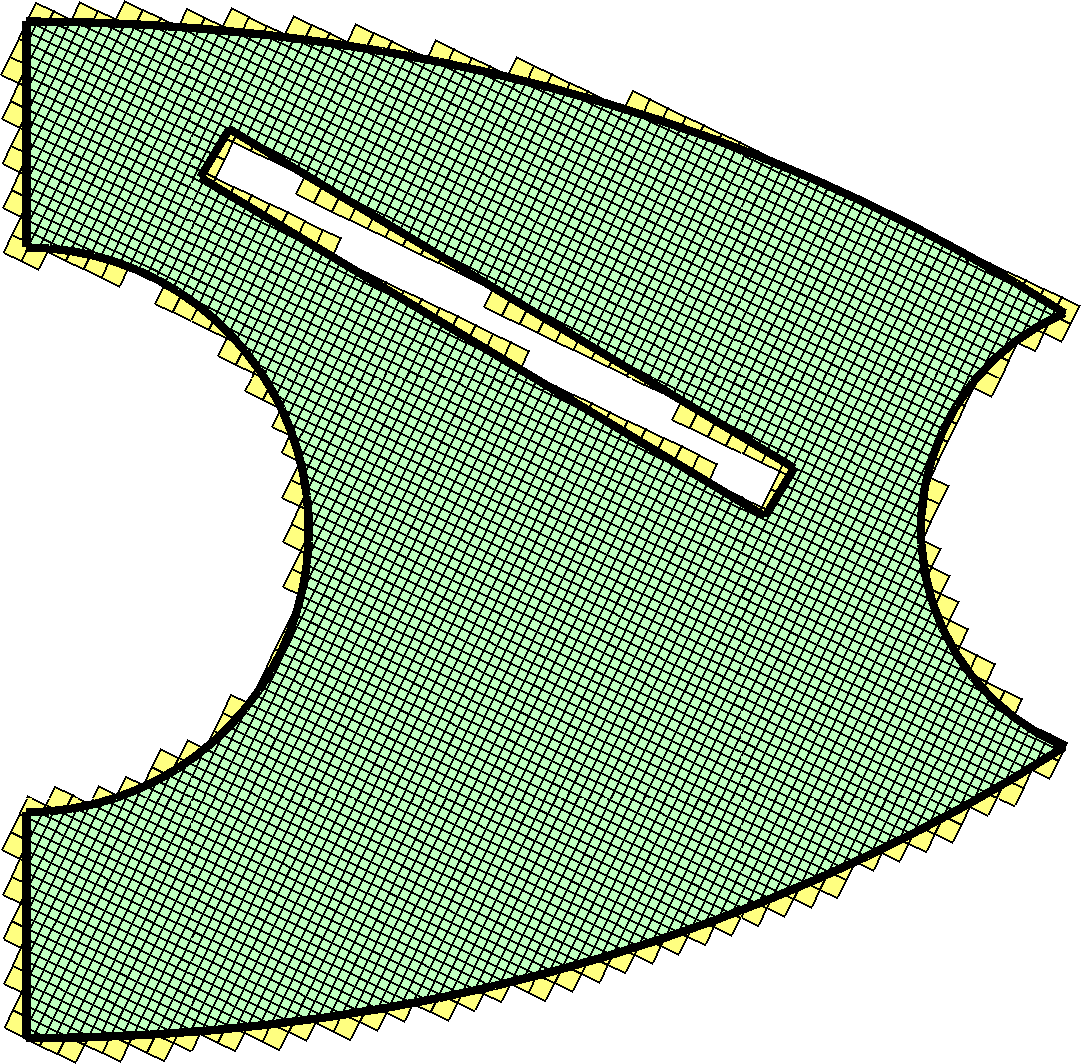}
\caption{Density resolution}
\label{fig:parametric-beam-density-mod}
\end{subfigure}

\begin{subfigure}[b]{\colRatioOne\textwidth}\centering
\includegraphics[width=\colRatioTwo\linewidth]{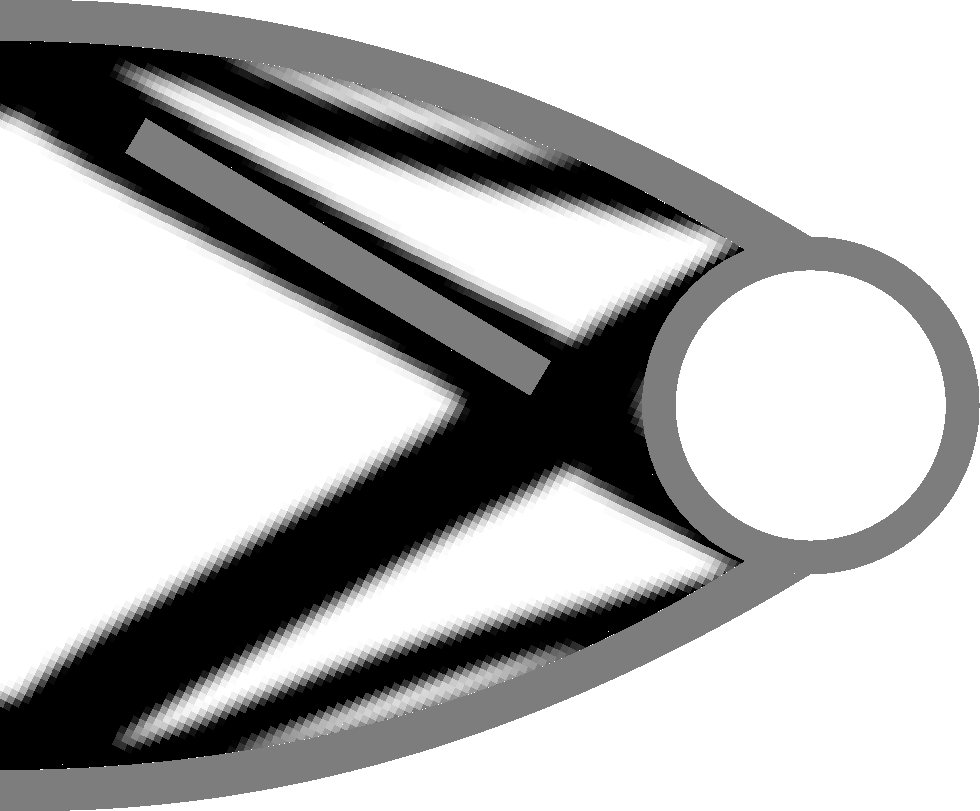}
\caption{Final geometry}
\label{fig:parametric-beam-final-mod}
\end{subfigure}
\begin{subfigure}[b]{\colRatioOne\textwidth}\centering
\includegraphics[width=\colRatioTwo\linewidth]{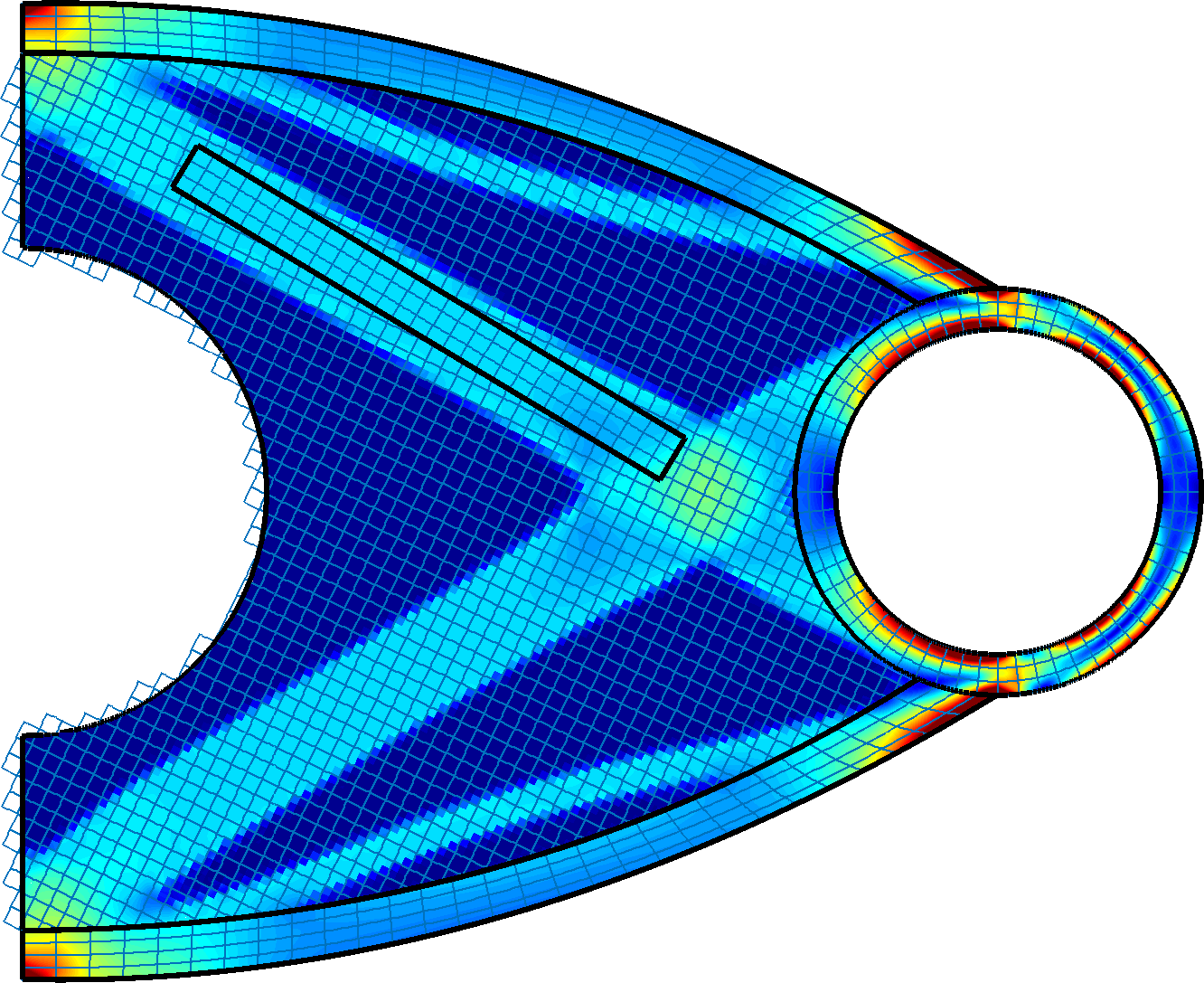}
\caption{Von-Mises stress in final geometry}
\label{fig:parametric-beam-stress-mod}
\end{subfigure}

\caption{
{Cantilever with parametric reinforcements (modified).}
Here a new nondesign domain region is placed inside the material part of the previous final geometry in Figure~\ref{fig:parametric-beam-final}.
}
\label{fig:parametric-beam-mod}
\end{figure}

\def\colRatioOne{1}
\def\colRatioTwo{0.6}
\begin{figure}\centering
\begin{subfigure}[b]{\colRatioOne\textwidth}\centering
\includegraphics[width=\colRatioTwo\linewidth]{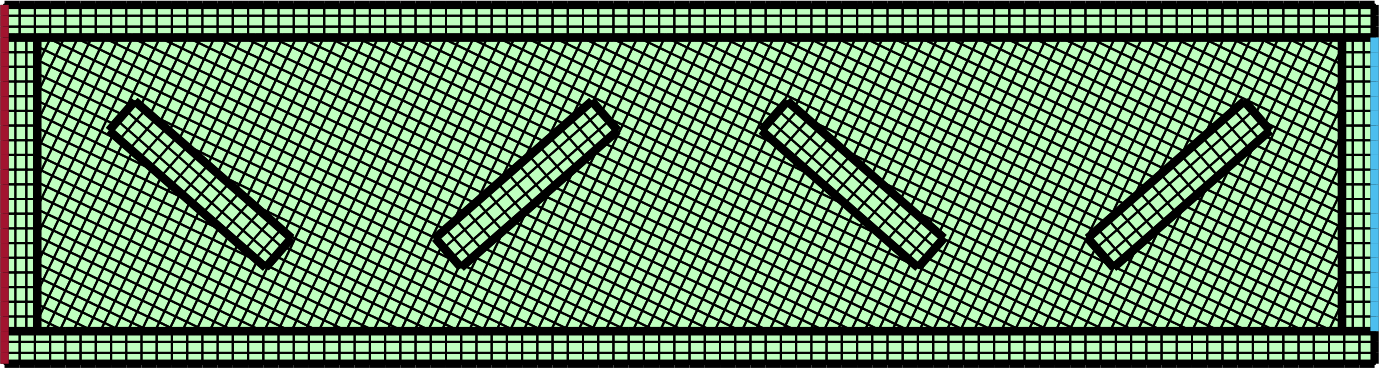}
\caption{Finite element meshes}
\label{fig:truss-structure-mesh}
\end{subfigure}
\begin{subfigure}[b]{\colRatioOne\textwidth}\centering
\includegraphics[width=\colRatioTwo\linewidth]{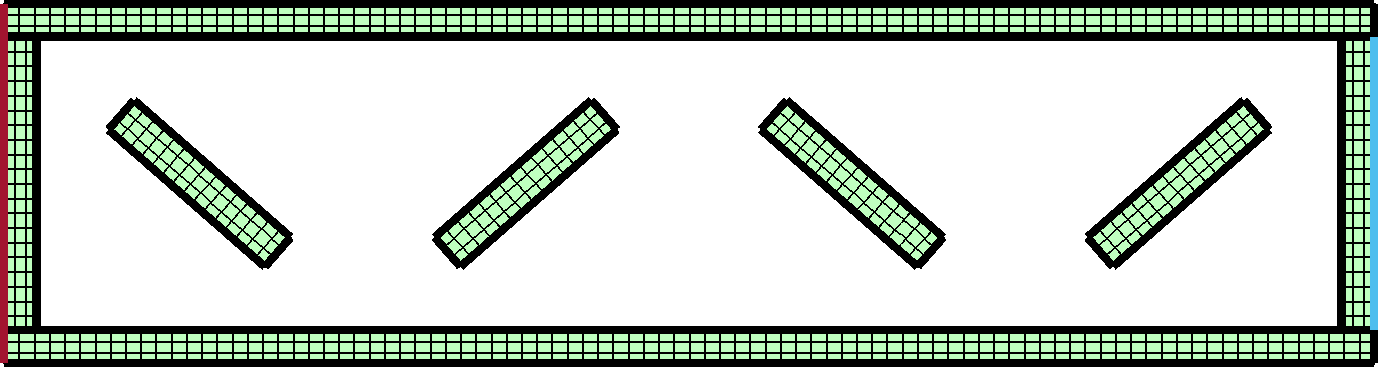}
\caption{Nondesign domain regions $\{\Omega_i\}_{i=1}^N$}
\label{fig:truss-structure-given}
\end{subfigure}
\begin{subfigure}[b]{\colRatioOne\textwidth}\centering
\includegraphics[width=\colRatioTwo\linewidth]{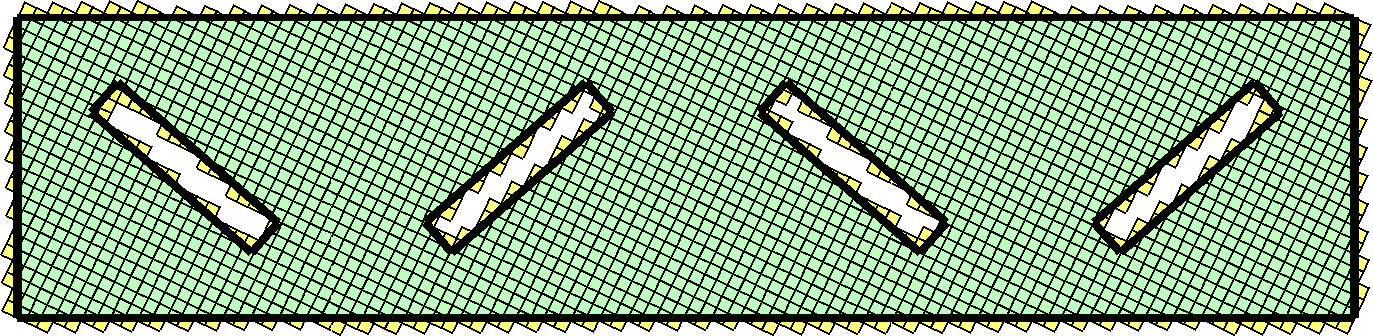}
\caption{Design domain $\Omega_0$}
\label{fig:truss-structure-design}
\end{subfigure}
\begin{subfigure}[b]{\colRatioOne\textwidth}\centering
\includegraphics[width=\colRatioTwo\linewidth]{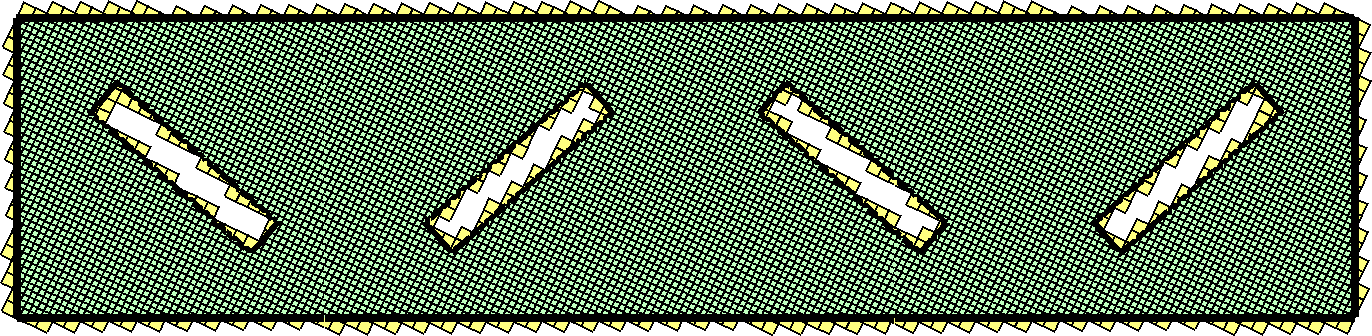}
\caption{Density resolution}
\label{fig:truss-structure-density}
\end{subfigure}
\begin{subfigure}[b]{\colRatioOne\textwidth}\centering
\includegraphics[width=\colRatioTwo\linewidth]{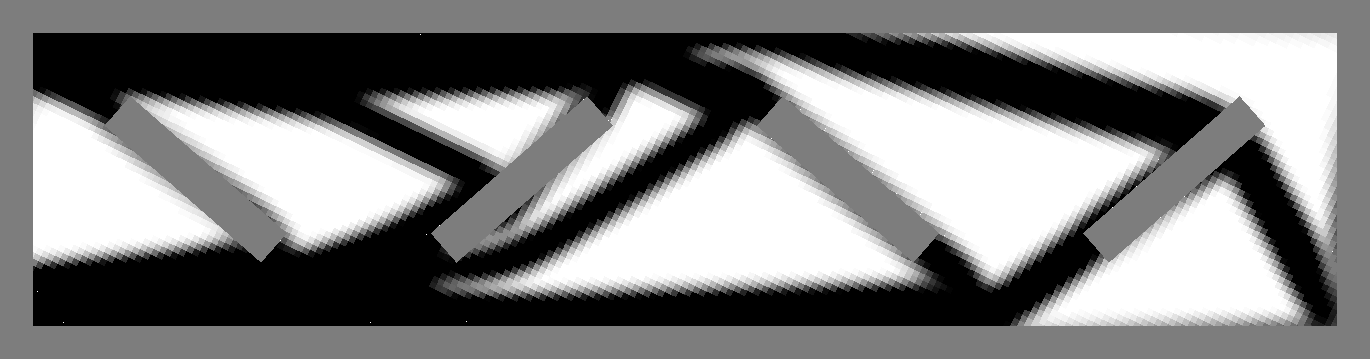}
\caption{Final geometry}
\label{fig:truss-structure-final}
\end{subfigure}
\begin{subfigure}[b]{\colRatioOne\textwidth}\centering
\includegraphics[width=\colRatioTwo\linewidth]{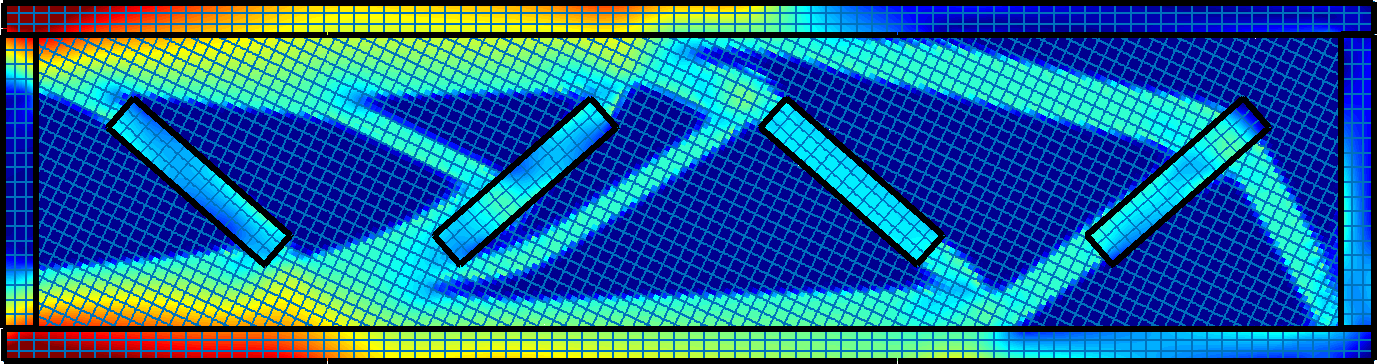}
\caption{Von-Mises stress in final geometry}
\label{fig:truss-structure-stress}
\end{subfigure}
\caption{
{Cantilever with truss structure reinforcements.}
Here the slanted nondesign domain regions in (b) internal to the outer frame are not positioned in an optimal way.
}
\label{fig:truss-structure}
\end{figure}

\section{Conclusions}
In this paper we have demonstrated the use of the cut finite element method in topology optimization.
The key feature of our approach is the \emph{flexible coupling to nondesign domain regions in a systematic and reliable manner}, which we have manifested using theoretical results and numerical examples.
To achieve this we have designed methods both for the imposition of Dirichlet boundary conditions on the design domain and for the coupling of the physical quantities between the design domain and nondesign domain regions that are robust with respect to the material density. The approach is based on a weighted Nitsche's method. To handle vanishing element cuts in a robust fashion we have proposed a ghost penalty stabilization, that may also be used to improve stability for vanishing material density in the design domain. The potential of the method was illustrated on three different computational examples. In this work we used tensor product B-splines for the discretization of the physical models, but other finite element methods can be applied in the same framework.

\bigskip
\bigskip
\noindent
\footnotesize {\bf Acknowledgments.}
This research was supported in part by the Swedish Foundation
for Strategic Research Grant No.\ AM13-0029 and the Swedish Research
Council Grants Nos.\  2013-4708, 2017-03911, 2018-05262. EB was supported by EPSRC research grants EP/P01576X/1 and EP/P012434/1.

\bibliographystyle{habbrv}
\footnotesize{
\bibliography{topopt}

\begin{thebibliography}{10}
\expandafter\ifx\csname url\endcsname\relax
  \def\url#1{\texttt{#1}}\fi
\expandafter\ifx\csname doi\endcsname\relax
  \def\doi#1{\burlalt{doi:#1}{http://dx.doi.org/#1}}\fi
\expandafter\ifx\csname urlprefix\endcsname\relax\def\urlprefix{URL }\fi
\expandafter\ifx\csname href\endcsname\relax
  \def\href#1#2{#2}\fi
\expandafter\ifx\csname burlalt\endcsname\relax
  \def\burlalt#1#2{\href{#2}{#1}}\fi

\bibitem{AJT04}
G.~Allaire, F.~Jouve, and A.-M. Toader.
\newblock Structural optimization using sensitivity analysis and a level-set
  method.
\newblock {\em J. Comput. Phys.}, 194(1):363 -- 393, 2004.
\newblock \doi{10.1016/j.jcp.2003.09.032}.

\bibitem{MR3047899}
C.~S. Andreasen and O.~Sigmund.
\newblock Topology optimization of fluid-structure-interaction problems in
  poroelasticity.
\newblock {\em Comput. Methods Appl. Mech. Engrg.}, 258:55--62, 2013.
\newblock \doi{10.1016/j.cma.2013.02.007}.

\bibitem{Andreassen2011}
E.~Andreassen, A.~Clausen, M.~Schevenels, B.~S. Lazarov, and O.~Sigmund.
\newblock {Efficient topology optimization in MATLAB using 88 lines of code}.
\newblock {\em Struct. Multidiscip. Optim.}, 43(1):1--16, 2011,
  \burlalt{9605103}{http://arxiv.org/abs/9605103}.
\newblock \doi{10.1007/s00158-010-0594-7}.

\bibitem{Bendsoe1989}
M.~P. Bends{\o}e.
\newblock {Optimal shape design as a material distribution problem}.
\newblock {\em Structural Optimization}, 1(4):193--202, 1989.
\newblock \doi{10.1007/BF01650949}.

\bibitem{Bendsoe95}
M.~P. Bends{\o}e.
\newblock {\em Optimization of structural topology, shape, and material}.
\newblock Springer-Verlag, Berlin, 1995.
\newblock \doi{10.1007/978-3-662-03115-5}.

\bibitem{Bendsoe1988}
M.~P. Bends{\o}e and N.~Kikuchi.
\newblock Generating optimal topologies in structural design using a
  homogenization method.
\newblock {\em Comput. Methods Appl. Mech. Engrg.}, 71(2):197--224, 1988.
\newblock \doi{10.1016/0045-7825(88)90086-2}.

\bibitem{MR2008524}
M.~P. Bends{\o}e and O.~Sigmund.
\newblock {\em Topology optimization: Theory, Methods, and Applications}.
\newblock Springer-Verlag, Berlin, 2004.
\newblock \doi{10.1007/978-3-662-05086-6}.

\bibitem{BerWadBer2018}
A.~Bernland, E.~Wadbro, and M.~Berggren.
\newblock Acoustic shape optimization using cut finite elements.
\newblock {\em Internat. J. Numer. Methods Engrg.}, 113(3):432--449, 2018.
\newblock \doi{10.1002/nme.5621}.

\bibitem{Bourdin2001}
B.~Bourdin.
\newblock Filters in topology optimization.
\newblock {\em Internat. J. Numer. Methods Engrg.}, 50(9):2143--2158, 2001.
\newblock \doi{10.1002/nme.116}.

\bibitem{Bruns2001}
T.~E. Bruns and D.~A. Tortorelli.
\newblock {Topology optimization of non-linear elastic structures and compliant
  mechanisms}.
\newblock {\em Comput. Methods Appl. Mech. Engrg.}, 190(26-27):3443--3459,
  2001.
\newblock \doi{10.1016/S0045-7825(00)00278-4}.

\bibitem{Burman2010}
E.~Burman.
\newblock Ghost penalty.
\newblock {\em C. R. Math. Acad. Sci. Paris}, 348(21-22):1217--1220, 2010.
\newblock \doi{10.1016/j.crma.2010.10.006}.

\bibitem{Burman2015}
E.~Burman, S.~Claus, P.~Hansbo, M.~G. Larson, and A.~Massing.
\newblock {CutFEM: Discretizing geometry and partial differential equations}.
\newblock {\em Internat. J. Numer. Methods Engrg.}, 104(7):472--501, 2015,
  \burlalt{1010.1724}{http://arxiv.org/abs/1010.1724}.
\newblock \doi{10.1002/nme.4823}.

\bibitem{Burman2017}
E.~Burman, D.~Elfverson, P.~Hansbo, M.~G. Larson, and K.~Larsson.
\newblock {A cut finite element method for the Bernoulli free boundary value
  problem}.
\newblock {\em Comput. Methods Appl. Mech. Engrg.}, 317:598--618, 2017.
\newblock \doi{10.1016/j.cma.2016.12.021}.

\bibitem{Burman2018}
E.~Burman, D.~Elfverson, P.~Hansbo, M.~G. Larson, and K.~Larsson.
\newblock {Shape optimization using the cut finite element method}.
\newblock {\em Comput. Methods Appl. Mech. Engrg.}, 328:242--261, 2018.
\newblock \doi{10.1016/j.cma.2017.09.005}.

\bibitem{Burman2012}
E.~Burman and P.~Hansbo.
\newblock Fictitious domain finite element methods using cut elements: {II}.
  {A} stabilized {N}itsche method.
\newblock {\em Appl. Numer. Math.}, 62(4):328--341, 2012.
\newblock \doi{10.1016/j.apnum.2011.01.008}.

\bibitem{MR3406629}
R.~E. Christiansen, B.~S. Lazarov, J.~S. Jensen, and O.~Sigmund.
\newblock Creating geometrically robust designs for highly sensitive problems
  using topology optimization.
\newblock {\em Struct. Multidiscip. Optim.}, 52(4):737--754, 2015.
\newblock \doi{10.1007/s00158-015-1265-5}.

\bibitem{MR3340167}
A.~Clausen, N.~Aage, and O.~Sigmund.
\newblock Topology optimization of coated structures and material interface
  problems.
\newblock {\em Comput. Methods Appl. Mech. Engrg.}, 290:524--541, 2015.
\newblock \doi{10.1016/j.cma.2015.02.011}.

\bibitem{MR3182450}
J.~D. Deaton and R.~V. Grandhi.
\newblock A survey of structural and multidisciplinary continuum topology
  optimization: post 2000.
\newblock {\em Struct. Multidiscip. Optim.}, 49(1):1--38, 2014.
\newblock \doi{10.1007/s00158-013-0956-z}.

\bibitem{ElfLarLar18}
D.~Elfverson, M.~G. Larson, and K.~Larsson.
\newblock Cut{IGA} with basis function removal.
\newblock {\em Adv. Model. Simul. Eng. Sci.}, 5(6):1--19, 2018.
\newblock \doi{10.1186/s40323-018-0099-2}.

\bibitem{2018arXiv180405654E}
D.~{Elfverson}, M.~G. {Larson}, and K.~{Larsson}.
\newblock A new least squares stabilized {N}itsche method for cut isogeometric
  analysis.
\newblock {\em Comput. Methods Appl. Mech. Engrg.}, 349:1--16, 2019.
\newblock \doi{10.1016/j.cma.2019.02.011}.

\bibitem{5617251}
A.~Erentok and O.~Sigmund.
\newblock Topology optimization of sub-wavelength antennas.
\newblock {\em IEEE T. Antenn. Propag.}, 59(1):58--69, Jan 2011.
\newblock \doi{10.1109/TAP.2010.2090451}.

\bibitem{MR3670454}
J.~P. Groen, M.~Langelaar, O.~Sigmund, and M.~Ruess.
\newblock Higher-order multi-resolution topology optimization using the finite
  cell method.
\newblock {\em Internat. J. Numer. Methods Engrg.}, 110(10):903--920, 2017.
\newblock \doi{10.1002/nme.5432}.

\bibitem{MR3620802}
L.~H\"{a}gg and E.~Wadbro.
\newblock Nonlinear filters in topology optimization: existence of solutions
  and efficient implementation for minimum compliance problems.
\newblock {\em Struct. Multidiscip. Optim.}, 55(3):1017--1028, 2017.
\newblock \doi{10.1007/s00158-016-1553-8}.

\bibitem{HaHa2002}
A.~Hansbo and P.~Hansbo.
\newblock An unfitted finite element method, based on {N}itsche's method, for
  elliptic interface problems.
\newblock {\em Comput. Methods Appl. Mech. Engrg.}, 191(47-48):5537--5552,
  2002.
\newblock \doi{10.1016/S0045-7825(02)00524-8}.

\bibitem{HaLaLa18}
P.~Hansbo, M.~G. Larson, and K.~Larsson.
\newblock Cut finite element methods for linear elasticity problems.
\newblock In {\em Geometrically unfitted finite element methods and
  applications}, volume 121 of {\em Lect. Notes Comput. Sci. Eng.}, pages
  25--63. Springer, Cham, 2017.
\newblock \doi{10.1007/978-3-319-71431-8_2}.

\bibitem{6750741}
E.~Hassan, E.~Wadbro, and M.~Berggren.
\newblock Topology optimization of metallic antennas.
\newblock {\em IEEE T. Antenn. Propag.}, 62(5):2488--2500, May 2014.
\newblock \doi{10.1109/TAP.2014.2309112}.

\bibitem{JonLarLar17}
T.~Jonsson, M.~G. Larson, and K.~Larsson.
\newblock Cut finite element methods for elliptic problems on multipatch
  parametric surfaces.
\newblock {\em Comput. Methods Appl. Mech. Engrg.}, 324:366--394, 2017.
\newblock \doi{10.1016/j.cma.2017.06.018}.

\bibitem{MR3038926}
A.~Klarbring and N.~Str{\"o}mberg.
\newblock Topology optimization of hyperelastic bodies including non-zero
  prescribed displacements.
\newblock {\em Struct. Multidiscip. Optim.}, 47(1):37--48, 2013.
\newblock \doi{10.1007/s00158-012-0819-z}.

\bibitem{MR2947650}
J.~Kook, K.~Koo, J.~Hyun, J.~S. Jensen, and S.~Wang.
\newblock Acoustical topology optimization for {Z}wicker's loudness
  model---application to noise barriers.
\newblock {\em Comput. Methods Appl. Mech. Engrg.}, 237/240:130--151, 2012.
\newblock \doi{10.1016/j.cma.2012.05.004}.

\bibitem{Li2001}
Q.~Li, G.~P. Steven, and Y.~M. Xie.
\newblock A simple checkerboard suppression algorithm for evolutionary
  structural optimization.
\newblock {\em Struct. Multidiscip. Optim.}, 22(3):230--239, 2001.
\newblock \doi{10.1007/s001580100140}.

\bibitem{MR3843862}
C.~Lundgaard, J.~Alexandersen, M.~Zhou, C.~S. Andreasen, and O.~Sigmund.
\newblock Revisiting density-based topology optimization for
  fluid-structure-interaction problems.
\newblock {\em Struct. Multidiscip. Optim.}, 58(3):969--995, 2018.
\newblock \doi{10.1007/s00158-018-1940-4}.

\bibitem{nitsche1971}
J.~Nitsche.
\newblock {\"U}ber ein {V}ariationsprinzip zur {L}{\"o}sung von
  {D}irichlet-{P}roblemen bei {V}erwendung von {T}eilr{\H a}umen, die keinen
  {R}andbedingungen unterworfen sind.
\newblock {\em Abh. Math. Sem. Univ. Hamburg}, 36:9--15, 1971.
\newblock \doi{10.1007/BF02995904}.

\bibitem{MR2878673}
J.~Parvizian, A.~D\"{u}ster, and E.~Rank.
\newblock Topology optimization using the finite cell method.
\newblock {\em Optim. Eng.}, 13(1):57--78, 2012.
\newblock \doi{10.1007/s11081-011-9159-x}.

\bibitem{Sigmund1997}
O.~Sigmund.
\newblock {On the design of compliant mechanisms using topology optimization}.
\newblock {\em Mech. Struct. Mach.}, 25(4):493--524, 1997.
\newblock \doi{10.1080/08905459708945415}.

\bibitem{Sigmund2007}
O.~Sigmund.
\newblock {Morphology-based black and white filters for topology optimization}.
\newblock {\em Struct. Multidiscip. Optim.}, 33(4-5):401--424, 2007.
\newblock \doi{10.1007/s00158-006-0087-x}.

\bibitem{Sigmund2013}
O.~Sigmund and K.~Maute.
\newblock Topology optimization approaches: A comparative review.
\newblock {\em Struct. Multidiscip. Optim.}, 48(6):1031--1055, 2013.
\newblock \doi{10.1007/s00158-013-0978-6}.

\bibitem{MR3646361}
C.~H. Villanueva and K.~Maute.
\newblock Cut{FEM} topology optimization of 3{D} laminar incompressible flow
  problems.
\newblock {\em Comput. Methods Appl. Mech. Engrg.}, 320:444--473, 2017.
\newblock \doi{10.1016/j.cma.2017.03.007}.

\bibitem{MR3240939}
E.~Wadbro.
\newblock Analysis and design of acoustic transition sections for impedance
  matching and mode conversion.
\newblock {\em Struct. Multidiscip. Optim.}, 50(3):395--408, 2014.
\newblock \doi{10.1007/s00158-014-1058-2}.

\bibitem{MR3395903}
E.~Wadbro and C.~Engstr\"{o}m.
\newblock Topology and shape optimization of plasmonic nano-antennas.
\newblock {\em Comput. Methods Appl. Mech. Engrg.}, 293:155--169, 2015.
\newblock \doi{10.1016/j.cma.2015.04.011}.

\bibitem{doi:10.1002/nme.2777}
G.~H. Yoon.
\newblock Topology optimization for stationary fluid–structure interaction
  problems using a new monolithic formulation.
\newblock {\em Internat. J. Numer. Methods Engrg.}, 82(5):591--616, 2010.
\newblock \doi{10.1002/nme.2777}.

\bibitem{Zhou1991}
M.~Zhou and G.~I.~N. Rozvany.
\newblock {The COC algorithm, Part II: Topological, geometrical and generalized
  shape optimization}.
\newblock {\em Comput. Methods Appl. Mech. Engrg.}, 89(1-3):309--336, 1991.
\newblock \doi{10.1016/0045-7825(91)90046-9}.

\end{thebibliography}
}

\bigskip
\bigskip
\noindent
\footnotesize {\bf Authors' addresses:}

\smallskip
\noindent
Erik Burman,  \quad \hfill \addressuclshort\\
{\tt e.burman@ucl.ac.uk}

\smallskip
\noindent
Daniel Elfverson,  \quad \hfill \addressumushort\\
{\tt daniel.elfverson@umu.se}

\smallskip
\noindent
Peter Hansbo,  \quad \hfill \addressjushort\\
{\tt peter.hansbo@ju.se}

\smallskip
\noindent
Mats G. Larson,  \quad \hfill \addressumushort\\
{\tt mats.larson@umu.se}

\smallskip
\noindent
Karl Larsson, \quad \hfill \addressumushort\\
{\tt karl.larsson@umu.se}

\end{document}